\documentclass[12pt]{article}
\usepackage{amsmath,fullpage,amsthm}
\usepackage{amscd}
\usepackage{amsthm,mathtools} \usepackage{amssymb}
\usepackage{latexsym}
\usepackage{eufrak}
\usepackage{euscript}
\usepackage{epsfig}
\usepackage{tikz}
\usepackage{graphics}
\usepackage{array}
\usepackage{enumerate} \usepackage{authblk}
\theoremstyle{plain}
\newtheorem{theorem}{Theorem}[section]
\newtheorem{corollary}{Corollary}[section]
\newtheorem{lemma}{Lemma}[section]

\newtheorem{remark}{Remark}[section]
\newtheorem{example}{Example}[section]

\newcommand{\PSL}{\mathop{\mathrm{PSL}}}
\newcommand{\PSU}{\mathop{\mathrm{PSU}}}
\newcommand{\Sz}{\mathop{\mathrm{Sz}}}
\newcommand{\PSp}{\mathop{\mathrm{PSp}}}
\newcommand{\SL}{\mathop{\mathrm{SL}}}
\newcommand{\SU}{\mathop{\mathrm{SU}}}

\begin{document}


  \setcounter{Maxaffil}{3}
   \title{On chordality of the power graph of finite groups}
   \author[ ]{Pallabi Manna\thanks{mannapallabimath001@gmail.com}}
   \author[ ]{Ranjit Mehatari\thanks{ranjitmehatari@gmail.com, mehatarir@nitrkl.ac.in}}
   \affil[ ]{Department of Mathematics,}
   \affil[ ]{National Institute of Technology Rourkela,}
   \affil[ ]{Rourkela - 769008, India}
   \maketitle
\begin{abstract}
A graph is called chordal if it forbids induced cycles of length $4$ or more. In this paper, we attempt to identify the non-nilpotent groups whose power graph is a chordal graph (this question was raised by Cameron in \cite{Cameron_1}). In this direction, we characterise the direct product of finite groups having chordal power graphs. We classify all finite simple groups of Lie type whose power graph is chordal. Further, we prove that the power graph of a sporadic simple group is always non-chordal. In addition, we show that almost all groups of order up to $47$ have chordal power graphs.
\end{abstract}


\section{Introduction}
In algebraic graph theory, we relish nice applications of abstract algebra in graph theory. In this regard, numerous graph structures on algebraic structures were defined by various academicians. Such graph structures add a very beautiful combinatorial flavour to abstract algebra. Group-derived graphs exhibit excellent network features, such as big girth, which is essential for assessing the network's connectedness. Additionally, these graphs have many applications in various research areas, including automata theory \cite{K_1, K_2}.  Beginning with the oldest graph, the Cayley (di)graph of a group, researchers have developed other graphs related to various algebraic structures, including power graphs, enhanced power graphs, commuting graphs, deep commuting graphs, generating graphs, conjugate graphs, coset graphs, non-commuting graphs, zero divisor graphs, intersection graphs, etc. However, we are only interested in the power graph of finite groups, which was first defined in 2002. This  type of graph was initially discovered as directed graphs of finite  semigroups by Kelarev and Quinn in \cite{Kelarev}. For a given semigroup $S$, the directed power graph of $S$ (denoted by $\overrightarrow{\mathcal{P}}(S)$) is a graph $(V,E)$ where $V=S$ and $E=\{(a,b): a \neq b,  b=a^{m}\text{ for some  }m\in \mathbb{N}\}$. The corresponding underlying graph was developed in 2009 by Chakrabarty \emph{et al.}~as the notion of undirected power graph of semigroups in \cite{Chakrabarty}. For a given semigroup $S$, the undirected power graph of $S$ (denoted by $\mathcal{P}(S)$) is a graph $(V,E)$ where $V=S$ and $E=\{\{a,b\}: a \neq b, \text{ and either }a=b^{m}\text{ or } b=a^{n},\text{ for some }m,n \in \mathbb{N}\}$. Later on, this definition was  adopted to describe the power graph of a finite group. More than 100 academic publications have been written that describe power graphs in various ways. Here, we have merely covered a few of the significant contributions made to the field of power graphs. 
\medskip

In \cite{Chakrabarty}, Chakrabarty \emph{et al.}~proved that if $G$ is a finite group, then $\mathcal{P}(G)$ is always connected and described a necessary and sufficient condition for the power graph $\mathcal{P}(G)$ of a finite group $G$ to be a complete graph. Additionally, there is a formula for counting the number of edges in $\mathcal{P}(G)$, given by $|E(\mathcal{P}(G))|=\frac{1}{2}[\sum_{a \in G}(2o(a)- \phi(o(a))-1)]$. In the same publication, we get the conclusion that  $\mathcal{P}(G)$ is Eulerian if and only if $G$ is a group of odd order.  \medskip

The proper power graph $\mathcal{P}^*(G)$ of a finite group $G$ is the graph obtained by deleting the identity element from $\mathcal{P}(G)$. In \cite{Curtin}, Curtin \emph{et al.}~determined the diameter of $\mathcal{P}^*(S_{n})$, where $S_n$ denotes the symmetric group on $n$ symbols. \medskip

The article \cite{Cameron_2} contains an interesting isomorphism theorem for power graphs of finite groups. In that paper, Cameron derived that if {$G_1$ and $G_2$} are two finite groups, then $\mathcal{P}(G_1) \cong \mathcal{P}(G_2)$ implies $\overrightarrow{\mathcal{P}}(G_1) \cong \overrightarrow{\mathcal{P}}(G_2)$.
Additionally, in \cite{Ghosh}, Cameron and Ghosh proved that  $\mathcal{P}(G_1) \cong \mathcal{P}(G_2)$ implies $G_1 \cong G_2$ for any two finite abelian groups {$G_1$ and $G_2$}. Further, they showed that if $G_1$ and $G_2$ are two finite groups with $\overrightarrow{\mathcal{P}}(G_1) \cong \overrightarrow{\mathcal{P}}(G_2)$, then $G_1$ and $G_2$ contain the same number of elements of each order. The same publication demonstrates that $Aut(G)=Aut(\mathcal{P}(G))$ holds only if $G$ is the Klein $4$-group. We have observed that, in general, an isomorphism between the power graphs of two groups may not result in an isomorphism between their directed power graphs. However, in \cite{Cameron_3}, Cameron showed that any isomorphism that preserves orientation induces an isomorphism between their directed power graphs by taking into consideration the torsion-free nilpotent groups of class $2$, especially $\mathbb{Z}$ and $\mathbb{Q}$. Zahirovi\'c \cite{ Samir} proved that $\mathcal{P}(G_1) \cong \mathcal{P}(G_2)$ implies $\overrightarrow{\mathcal{P}}(G_1) \cong \overrightarrow{\mathcal{P}}(G_2)$ for any two torsion-free groups $G_1$ and $G_2$.
For more interesting results related to power graphs, we refer to two survey papers \cite{Abawajy, Kumar}.\medskip

Several significant graph classes can be defined in terms of induced forbidden subgraphs. These graph classes {include} of threshold graphs, split graphs, cographs, and chordal graphs, among others. In a previous paper \cite{Manna}, we started our discussion about the various forbidden induced subgraphs of power graphs. While only partial results were developed for cographs and chordal graphs, we have successfully identified all groups whose power graphs are split and threshold. Also, we characterised all finite nilpotent groups whose power graphs are cographs and chordal graphs. We generalised a little more in \cite{Mehatari} by attempting to address the open problem, ``classify all the finite groups whose power graph is a cograph." \medskip

 A graph is said to be \textit{chordal} if any cycle of length four and more has a chord, or in terms of forbidden subgraphs, we can define a chordal graph as a graph with no induced cycle of length more than $3$. Identifying whether a graph is chordal or not is always a difficult task. If a graph is a cograph, it does not necessarily mean that it is a chordal graph; in order to determine whether a cograph is chordal, we {need only determine} whether it has any cycle of length four. If a graph is not a cograph, however, difficulties arise.\medskip

It is clear that if a graph $\Gamma$ is chordal, then all its induced subgraphs are also chordal. So, it follows immediately that if a graph $\Gamma$ contains a non-chordal induced subgraph, then $\Gamma$ is also a non-chordal graph. Thus, we can say the chordality of a graph is subgraph-closed. The subgraph-closed property of chordality plays a crucial role in the case of a power graph. It is well known that if $G$ is a finite group and $H$ is a subgroup of $G$, then $\mathcal{P}(H)$ is an induced subgraph of $\mathcal{P}(G)$. Thus, to draw a conclusion about the non-chordality of the power graph of a finite group $G$, it suffices to show that $G$ contains a subgroup whose power graph is not a chordal graph.

We will use the term power-chordal group for a finite group whose power graph is a chordal graph. Analogously, we call a group $G$ non-power-chordal if $G$ is not power-chordal.\medskip

{The goal of this paper is} to locate non-nilpotent groups whose power graph is a chordal graph (a question was raised by Cameron in \cite{Cameron_1}). We only partially address this problem in this paper. {This paper is organized as follows.}

In Section 2, we review some significant definitions and findings, which we will use later on. In Section 3, we completely identify all possible finite groups $G$ and $H$ for which the power graph of $G \times H$ is a chordal graph. In Section 4, we prove that the power graph of $S_n$ is chordal if and only if $n\leq5$. In Section 5, finite simple groups are considered. In this section, we prove that no sporadic simple group is power-chordal. Additionally, we give a complete characterization for the classes of all finite {simple group of Lie type} whose power graph is a chordal graph. In this regard, we obtain: if $G$ is a finite {simple group of Lie type}, then $\mathcal{P}(G)$ is a chordal graph if $G$ is any one of the following groups. 
\begin{itemize}
    \item 
    $G=\PSL(2,q)$, where $q$ is a prime power and {$q$ is odd}, such that both $(q\pm 1)/2$ are either a power of prime or of the form $p_{1}^{a}p_{2}$ for distinct primes $p_{1}$, $p_{2}$;
    \item 
    $G=\PSL(2,q)$, where $q$ is a prime power and $q$ is even, such that both $q\pm 1$  are either a power of prime or of the form $p_{1}^{a}p_{2}$ for distinct primes $p_{1}$, $p_{2}$;
\item 
$G=\Sz(q)$ with $q=2^{2e+1}$ such that $q-1, q+\sqrt{2q}+1$ and $q-\sqrt{2q}+1$ are either prime power or of the form $p_{1}^{a}p_{2}$ for distinct primes $p_{1}$, $p_{2}$;
\item
 $G={}^{2}{G_{2}}(q)=R_{1}(q)$ where $q=3^{2e+1}$, for infinitely many values of $q$;
\item 
$\PSL(3,2)$, $\PSL(3,4)$.
\end{itemize}
In the last section, we compute certain finite groups of small order whose power graph is a chordal graph. This section also provides some nice applications of the results obtained in sections 3 and 4. Here our major observation is that if $G$ is the group {of order at most $47$} except $24$, $30$, $36$, $40$, and $42$, then $\mathcal{P}(G)$ is always a chordal graph, whereas there are several groups with order $24$, $30$, $36$, $40$, or $42$ whose power graph is not a chordal graph.


\section{Preliminaries}
In this section, we will review some fundamental definitions, terminologies, and {results} that will be used in subsequent sections. We only consider finite groups in this study. We use the notation $C_{n}$ to represent the cyclic group of order $n$ as well as {the graph which is a cycle of length $n$}; the context will make this notation apparent. Also, we use $D_{n}$ to denote the dihedral group of order $2n$. 
Now we recall two theorems from our previous paper \cite{Manna} regarding the chordality of the power graph of finite nilpotent groups.
\begin{theorem}
\cite{Manna}
\label{Chordal_pre_th1}
The power graph of a finite $p$-group is a chordal graph.
\end{theorem}

\begin{theorem}
\cite{Manna}
\label{Chordal_pre_th5}
Let $G$ be a finite nilpotent group which is not a $p$-group. Then $\mathcal{P}(G)$ is a chordal graph if and only if $|G|$ has exactly two distinct prime divisors  such that one of two Sylow subgroups is cyclic and the other one has prime exponent.
\end{theorem}

\begin{corollary}
\label{Chordal_pre_cor1}
$\mathcal{P}(C_{n})$ is chordal if and only $n$ is either a prime power or a product of a prime and a prime power.
\end{corollary}

\begin{corollary}
\label{Chordal_pre_cor2}
$\mathcal{P}(D_{n})$ is chordal if and only $n$ is either a prime power or a product of a prime and a prime power.
\end{corollary}

\begin{proof} 
As $\mathcal{P}^*(D_n)$ contains $n$ isolated vertices along with a component isomorphic to $\mathcal{P}^*(C_n)$, so $\mathcal{P}(D_n)$ is chordal if and only if $\mathcal{P}(C_n)$ is chordal. Hence the result follows.
\end{proof}
\medskip

A group $G$ is called an \textit{EPPO} group if every non-identity element of $G$ is of prime power order, and $G$ is called an \textit{EPO} group if every non-identity element of $G$ is of prime order. If $G$ is a finite group, then the Gruenberg-Kegel graph (or GK graph or prime graph) is a graph whose vertices are distinct prime divisors of $|G|$, and two distinct vertices $p$ and $q$ are joined by an edge if $G$ contains an element of order $pq$. Now we recall some results that connect EPPO groups with the corresponding GK graph and power graph (see \cite{Cameron, Ma,Manna}).
We define the enhanced power graph of a group before moving on to the next theorems.

Let $G$ be a group. The enhanced power graph of $G$ is the graph whose vertex set is $G$ and two distinct vertices $x, y$ are adjacent if they belong to the same cyclic subgroup (i.e., $\langle x, y \rangle$ is a cyclic group).

\begin{theorem}
\cite{Cameron}
\label{Chordal_pre_th2}
For a finite group $G$, the following conditions are equivalent:
\begin{enumerate}
    \item $G$ is an EPPO group.
    \item The GK graph of $G$ is a null graph.
    \item $\mathcal{P}(G)$ is equal to the enhanced power graph of $G$.
\end{enumerate}
\end{theorem}

\begin{theorem}
\cite{Ma}
\label{Chordal_pre_th3}
The enhanced power graph of an EPPO group  is chordal.
\end{theorem}

Thus, as a consequence of above two theorems, we obtain the following result 
\begin{corollary}
\label{Chordal_pre_cor4}
If $G$ is an EPPO group then following statements hold.
 \begin{enumerate}
\item The  GK graph of $G$ is a null graph.
\item The power graph of $G$ is a chordal graph.
 \end{enumerate}
\end{corollary}

\begin{example}
If $G$ is the group $C_{p} \rtimes C_{q^{m}}$, where $p$ and $q$ are primes with $p>q$ and $q^{m}|(p-1)$, then the power graph of $G$ is a chordal graph because $G$ does not contain any element of order $pq$. So, it is an EPPO group, and its GK graph is a null graph.
\end{example}


\section{Direct product of two groups}
\label{direct products}
In this section, we describe all finite groups $G$ and $H$ for which the power graph of $G \times H$ is chordal. If both $|G|$ and $|H|$ are powers of the same prime, then $G\times H$ is a $p$-group, and hence its power graph is a chordal graph by Theorem \ref{Chordal_pre_th1}. In Subsection 3.1, we consider the scenario where the order of $G\times H$ has exactly two distinct prime divisors, while in Subsection 3.2, we consider the case where $|G\times H|$ has three or more distinct prime divisors. We start by presenting a few introductory lemmas.

\begin{lemma}
\label{Cordal_product_lm2}
Let $p$ be a prime and $H$ be a finite group whose every non-identity element has order either a power of $p$ or a prime other than $p$.  Then $\mathcal{P}(C_{p} \times H)$ is a chordal graph.
\end{lemma}

\begin{proof}
For the sake of contradiction, let $\cdots\sim x_1\sim x_2\sim x_3\sim x_4\sim \cdots$ be an induced cycle of length $\geq 6$. Without loss of generality, we assume that $\cdots\leftarrow x_{1}\rightarrow x_{2}\leftarrow x_{3}\rightarrow x_{4}\leftarrow\cdots$ in
$\overrightarrow{\mathcal{P}}(C_{p} \times H)$. Let $x_{i}=(g_{i}, h_{i})$, where $g_i\in C_{p}$, $h_i\in H$.
If $o(g_3)=1$, then clearly $x_{2}\sim x_{4}$, so $o(g_3)\neq1$; and for a similar reason $o(h_3)\neq1$. Thus we must have $o(g_3)=p$. By assumption, $o(h_3)$ is either a power of $p$ or a prime $q\neq p$. We consider the following cases.\medskip\\
\textbf{Case I.} Let $o(g_{3})=p$ and $o(h_3)=p^{k}$. Then $o(g_2), o(g_4)\in\{1,p\}$ and $o(h_2), o(h_4)\in\{1,p,p^2,\ldots,p^k\}$. This implies that orders of $x_2,x_3$ and $x_4$ are  powers of $p$. Now, since $x_2$ and $x_4$ are elements of the same cyclic $p$-group $\langle x_3\rangle$, so $x_2$, $x_3$ and $x_4$ form a triangle, which is a contradiction.\medskip\\
\textbf{Case II.} Let $o(g_3)=p$ and $o(h_3)=p_{j}$ for some $p_j\neq p$.
 If $o(g_2)=p=o(g_4)$ and $o(h_2)=p_j= o(h_4)$, then the vertices $x_2$, $x_3$ and $x_4$ form a triangle in the power graph. So without loss of generality, we take $o(g_2)=p$, $o(h_2)=1$, and $o(g_4)=1$, $o(h_4)=p_{j}$. Then we obtain $x_{3}\sim x_{5}$, which gives a contradiction.

Thus, $\mathcal{P}(C_p\times H)$ does not contain an induced cycle of length $\geq 6$. A similar argument runs for cycles of length 4. Hence $\mathcal{P}(C_{p}\times H)$  is a chordal graph.
\end{proof}
By using a similar argument, one can prove the following result:
\begin{lemma}
\label{Cordal_product_lm3}
Let $H$ be a finite group in which every non-identity element has prime order. Then $\mathcal{P}(C_{p^{\alpha}}\times H)$ is power-chordal.
\end{lemma}


\subsection{Direct product of $G$ and $H$, where $|G\times H|$ has exactly two distinct prime divisors}
We are now ready to identify the finite groups $G$ and $H$, with $|G\times H|$ having exactly two distinct prime divisors, such that $\mathcal{P}(G\times H)$ is a chordal graph. We observe that if $G$ is a $p$-group and $H$ is a $q$-group, then $\mathcal{P}(G\times H)$ is chordal if and only if one of the groups is cyclic and the other has a prime exponent. In the next few theorems, we consider that at least one of $G$ or $H$ is not nilpotent. We use the Fitting subgroup of $G\times H$ to examine the chordality of $\mathcal{P}(G\times H)$. The Fitting subgroup of a group $G$, denoted by $\mathrm{Fit}(G)$, is the unique largest normal nilpotent subgroup. Let $G$ be a finite group and $p$ be a prime divisor of $|G|$. For any prime divisor $p$ of $|G|$, we use $O_{p}(G)$ to denote the largest normal $p$-subgroup of $G$. The following theorem is the main result of this section.

\begin{theorem}
\label{Chordal_product_th}
Let $G, H$ be two groups such that $|G\times H|$ has exactly two distinct prime divisors  {$p$ and $q$}. Then $\mathcal{P}(G\times H)$ is chordal if and only if {one of the following holds.}
\begin{itemize}
\item[(i)] {One} of $G$ and $H$ (say $G$) is the cyclic group $C_{p^{\alpha}}$ with $\alpha \geq 1$ and the other (say $H$) has one of the following structures {for some $r, s, m ,n , k \geq 1$}:
\begin{enumerate}
\item
$O_{p}(H)\rtimes C_{q^s}$ or $C_{q^s}\rtimes O_{p}(H)$.
\item
$Q\rtimes C_{p^r}$ or $C_{p^r}\rtimes Q$, where $Q$ is a $q$-group of exponent $q$.
\item
 $(O_{p}(H)\times O_{q}(H))\rtimes C_{p^m}$ or $C_{p^m}\rtimes (O_{p}(H)\times O_{q}(H))$.
 \item
  $(O_{p}(H)\times O_{q}(H))\rtimes C_{q^n}$ or $C_{q^n}\rtimes (O_{p}(H)\times O_{q}(H)$.  \\
For (c) and (d), either $O_{p}(H)$ is a $p$-group of exponent $p$ and $O_{q}(H) \cong C_{q^k}$, or $O_{q}(H)$ is a non-cyclic $q$-group of exponent $q$ and $O_{p}(H)$ is a cyclic $p$-group. \medskip 

  Moreover, the following two additional possibilities exist for $H$ if $\alpha=1:$
\item 
$C_{q^{s}}\rtimes C_{p^r}$ or $C_{p^r}\rtimes C_{q^{s}}$.
\item
$(O_{p}(H)\times C_{q^{s}})\rtimes C_{p^m}$ or $C_{p^m}\rtimes (O_{p}(H)\times C_{q^{s}})$,where $O_{p}(H)$ is a $p$-group of exponent $p$.
\end{enumerate}
\item[(ii)] $G\cong C_{pq}$ and $H$ is an EPO group.
\item[(iii)] {One} of $G$ and  $H$ is a non-cyclic $p$-group of exponent $p$ and the  other is one of the groups: $C_{q^s}\rtimes C_{p^r}$, $C_{p^r}\times C_{q}$, $C_{p}\times C_{q^s}$ ($r, s \geq 1$).
 \item[(iv)] {$G$ and $H$} must satisfies {(1) or (2) below}.\\
(1) {Both} $G$ and $H$ are EPPO groups such that there exists some prime (say, $p$) dividing both $|G|$ and $|H|$ with Sylow $p$-subgroups (of both {$G$ and $H$}) are cyclic as well as normal, and $G$ and $H$ have either of the following special form: 
\begin{enumerate}
\item
 At least one of $G$ and $H$ is an EPO group.
\item
 Both $G$ and $H$  have elements of orders either $q$ or power of $p$.
\end{enumerate}
(2) $G$ is an EPPO group and $H$ is a non-EPPO group such that they have the following forms:\\
(x) $G\cong C_{P}\rtimes Q$ and $H \cong C_{pq}\rtimes C_{q^{s-1}}$;\\
(y) $G \cong C_{q}\rtimes P$ and $H\cong C_{pq}\rtimes C_{p^{r-1}}$; \\
(z) $G \cong C_{p}\rtimes C_{q}$ and $H\cong C_{pq}\rtimes C_{p^{r-1}}$ or $G\cong C_{q}\rtimes C_{p}$ and $H\cong C_{pq}\rtimes C_{q^{s-1}}$.
Here, $r, s \geq 2$, and $P, Q$ are the non-cyclic $p$-group of exponent $p$ and the non-cyclic $q$-group of exponent $q$ respectively.
\end{itemize}
\end{theorem}

\medskip

We prove this theorem by proving the following lemmas:
\begin{lemma}
\label{Chordal_product_th1}
{Let $p$ and $ q$ be two primes. Let $G$ be the cyclic group $C_{p^{\alpha}}$ with $\alpha \geq 1$ and $H$ be a non-nilpotent group of order $p^{r}q^{s}$ ($r, s \geq 1$)}. Then $\mathcal{P}(G\times H)$ is chordal if and only if $H$ has one of the structures described in Theorem \ref{Chordal_product_th}(i) (a-f).
\end{lemma}

\begin{proof}
{Suppose $\mathcal{P}(G\times H)$ is chordal. Let $P_{H}$ and $Q_{H}$ be  Sylow $p$- and 	 Sylow $q$-subgroups of $H$, respectively}. Here we consider two cases based on the value of $\alpha$.\medskip\\
\textbf{Case 1.} $\alpha >1$. \\
{Here} $G\times Q_{H}$ is nilpotent and power-chordal. Thus, by Theorem \ref{Chordal_pre_th5}, $Q_{H}$ is a $q$-group of exponent $q$ (since $G$ is cyclic). Now, consider the Fitting subgroup $\mathrm{Fit}(H)$ of $H$.

 If $\mathrm{Fit}(H)$ is a $p$-group, then $\mathrm{Fit}(H)\cong O_{p}(H)$. Thus, either $H\cong (O_{p}(H)\rtimes C_{q^s})$ or $H\cong C_{q^s}\rtimes O_{p}(H)$. 
 
 Similarly, if $\mathrm{Fit}(H)$ is a $q$-group, then we obtain that $H\cong Q\rtimes C_{p^r}$ or $H\cong C_{p^r}\rtimes Q$, where $Q$ is a $q$-group of exponent $q$.
 
If $pq$ divides $|\mathrm{Fit}(H)|$, then $\mathrm{Fit}(H)\cong (O_{p}(H)\times O_{q}(H))$. As $\mathrm{Fit}(H)$ is nilpotent and power-chordal, so $H$ must be one form (i)(a), (i)(b), (i)(c), or (i)(d) of Theorem \ref{Chordal_product_th}. \medskip\\
\textbf{Case 2.} $\alpha =1$.\\ 
Since $G\times Q_{H}$ is nilpotent and power-chordal, Theorem \ref{Chordal_pre_th5} gives that $Q_{H}$ is either a $q$-group of exponent $q$ or $C_{q^{s}}$.
If $Q_{H}$ is a $q$-group of exponent $q$, then we get the structure of $H$ as any one of (i)(a), (i)(b), (i)(c), or (i)(d) of Theorem \ref{Chordal_product_th}. 

Next, suppose that $Q_{H}\cong C_{q^{s}}$. If $\mathrm{Fit}(H)$ is a $p$-group then by a similar approach as in Case 1, we get any one of the structures for $H$ as (i)(a) in Theorem \ref{Chordal_product_th}.
 
 If $\mathrm{Fit}(H)$ is a $q$-group, then $\mathrm{Fit}(H)\cong C_{q^{s}}$ and hence $H \cong C_{q^{s}}\rtimes C_{p^r}$ or $H\cong C_{p^r}\rtimes C_{q^s}$. This gives the structure of $H$ as  (i)(e) in Theorem \ref{Chordal_product_th}.

 Finally, if $pq$ divides $|\mathrm{Fit}(H)|$, then $\mathrm{Fit}(H)\cong (O_{p}(H)\times C_{q^{s}})$. Therefore $H\cong (O_{p}(H)\times C_{q^{s}})\rtimes C_{p^m}$ or $C_{p^m}\rtimes (O_{p}(H)\times C_{q^s})$. Since $\mathrm{Fit}(H)$ is nilpotent and power-chordal, so $O_{p}(H)$ must be a $p$-group of exponent $p$ (by Theorem \ref{Chordal_pre_th5}).\medskip

To prove the converse part of the theorem, we consider the following cases.\medskip\\
\textbf{Case I.} Let $H$ be a group of the form  (i)(a) of Theorem \ref{Chordal_product_th}. Now, if $H$ is an EPPO group, then every non-identity element of $H$ is of order either a power of $q$ or a power of $p$. Let $\cdots\leftarrow x_{1}\rightarrow x_{2}\leftarrow x_{3}\rightarrow x_{4}\leftarrow\cdots$ be a cycle in $\overrightarrow{\mathcal{P}}(C_{p^{\alpha}} \times H)$, where $x_{i}=(g_{i}, h_{i})$ (say). If $o(g_{3})=1$, then $x_{2}\sim x_{4}$; so $o(g_{3})\neq 1$. For a similar reason $o(h_{3})\neq 1$. Let $g_{3}=a$, $h_{3}=b$ where $o(b)$ is a power of $p$. Then $x_{2}\sim x_{4}$, i.e., $x_{2}, x_{3}, x_{4}$ form a triangle. Now, if we take $g_{3}=a$ and $h_{3}=b^q$ with $o(b)=q^k$, then by a similar approach as in Lemma \ref{Cordal_product_lm2}, we obtain either $x_{1}, x_{2}, x_{3}$ or $x_{3}, x_{4}, x_{5}$ form a triangle. Thus, in this case, $\mathcal{P}(C_{p^{\alpha}}\times H)$ is chordal. On the other hand, if $H$ is not an EPPO group, then it contains an element, say $h\in H$, of order $pq$. Without loss of generality, we assume that $g_{3}=a$, $h_{3}=h^{p}$ where $o(a)=p$, then either $g_{2}=1$, $h_{2}=h^{p^2}$, $g_{4}=a^{q}$, $h_{4}=1$ or $g_{2}=a^{q}, h_{2}=1$, $g_{4}=1, h_{4}=h^{p^2}$. So either $h_{1}\in \langle h\rangle$ or $h_{5}\in \langle h\rangle$. As $G$ is cyclic, we conclude that either $x_1\sim x_3$ or $x_{3}\sim x_{5}$. Hence, $\mathcal{P}(G\times H)$ is a chordal graph.\medskip\\
\textbf{Case II. } Let $H$ be any one of the form Theorem \ref{Chordal_product_th} (i)(b). Then, by a similar approach as in Case I, we can prove that $\mathcal{P}(G\times H)$ is chordal.\medskip\\
\textbf{Case III. } {If $H$ is any of the groups in} (i)(c) or (i)(d) in Theorem \ref{Chordal_product_th}, then every non-identity element of $H$ is of order either $p^{k}q^{m}$ or $p^s$ or $q^n$ for some $k, m,  s, n\geq 1$. Let $\cdots x_{1}\rightarrow x_{2}\leftarrow x_{3}\rightarrow x_{4}\cdots$ be a cycle in $\overrightarrow{\mathcal{P}}(C_{p^{\alpha}}\times H)$, where $x_{i}=(g_i, h_i)$.

Let $g_{3}=a$, $h_{3}=b$. If both $o(a)$ and $o(b)$ are powers of the same prime, then $x_{2}\sim x_{4}$. Suppose $o(a)=p$ and $o(b)=q^n$. Let $g_{3}=a$, $h_{3}=b$. Then, we must have either  $g_{2}=1, h_{2}=b^{p}$, $g_{4}=a^{q^n}, h_{4}=1$ or $g_{2}=a^{q^n}, h_{2}=1$, $g_{4}=1, h_{4}=b^{p}$. Then, either $h_{1}\in \langle b\rangle$ or $h_{5}\in \langle b\rangle$. Now, since $G\cong C_{p^\alpha}$,  either $x_{1}\sim x_{3}$ or $x_{5}\sim x_{3}$. 

Next, we assume that $o(a)=p$ and $o(b)=p^{k}q^{m}$. Let $g_{3}=a$, $h_{3}=b^{p^k}$, then again we get either $x_{3}\sim x_{5}$ or $x_{1}\sim x_{3}$. This proves that $\mathcal{P}(C_{p^\alpha}\times H)$ is a chordal graph.\medskip\\
\textbf{Case IV.} If $H$ is a group of the form as in  Theorem \ref{Chordal_product_th} (i)(e), then proceeding as in Case I, we can prove that $\mathcal{P}(G\times H)$ is a chordal graph.\medskip\\
\textbf{Case V.} Let $H$ be of the form  as (i)(f) in Theorem \ref{Chordal_product_th}. Then, by following an approach similar to Case III, one can easily prove that $\mathcal{P}(G\times H)$ is a chordal graph. 
This completes the proof of the theorem.
\end{proof}

\begin{lemma}
\label{Chordal_product_th2}
Let $G$ and $H$ be two finite groups with $|G\times H|$ has exactly two distinct prime divisors $p$ and $q$, and {suppose} $pq$ divides both $|G|$ and $|H|$. If $G$ is cyclic, then $\mathcal{P}(G\times H)$ is a chordal graph if and only if $G\cong C_{pq}$ and $H$ is an EPO group.
\end{lemma}

\begin{proof}
Let $G \cong C_{p^{\alpha}q^{\beta}}$ and $|H|=p^{k}q^{m}$ where $ k,m, \alpha, \beta \geq 1$. Let $P_1$, $Q_1$, and $P_2$, $Q_2$ be Sylow $p$- and Sylow $q$-subgroups of $G$ and $H$, respectively. Now, $G \times P_{2}$ is nilpotent (as $G$ is cyclic and $P_{2}$ is a $p$-group); so, by Theorem \ref{Chordal_pre_th5}, $\alpha=1$ and $P_2$ is a $p$-group of exponent $p$. On the other hand, since $G \times Q_{2}$ is nilpotent, $\beta =1$ and $Q_2$ is a $q$-group of exponent $q$. Therefore, $G \cong C_{pq}$.
Also, we observe that $H$ must be an EPPO group, or else $G\times H$ contains the non-power-chordal subgroup $C_{pq}\times C_{pq}$. Therefore, $H$ is an EPO group. \medskip

{For the converse, suppose $G$ and $H$ have the given structure}. Let $\cdots\leftarrow x_{1}\rightarrow x_{2}\leftarrow x_{3}\rightarrow x_{4}\leftarrow\cdots$ be a cycle  in 
$\overrightarrow{\mathcal{P}}(G \times H)$, where $x_{i}=(g_{i}, h_{i})$ (say). Let $g_{3}=a$, $h_{3}=b$, where  $o(a)\in \{1, p, q, pq\}$ and $o(h_{3})\in \{1, p, q\}$.

If $o(a)=1$, then $x_{2}\sim x_{4}$. Next, let $o(b)=1$ and $o(a)\neq 1$. Without loss of generality, we consider $g_{2}=c^{p}, h_{2}=1$, $g_{4}=c^{q}, h_{4}=1$. So an induced {cycle of length at least $4$ is not possible} in $\mathcal{P}(G\times H)$.

Next, suppose that $o(a)\neq 1$ and $o(b)=p$. In this case, $o(a)\in \{p, q, pq\}$; but in every situation we obtain either $x_{3}\sim x_{5}$ or $x_{1}\sim x_{3}$ or $x_{2}\sim x_{4}$. The same argument works when $o(a)\neq 1$ and $o(b)=q$. Thus, $\mathcal{P}(G\times H)$ does not induce a cycle of length $4$ or more. Hence, $\mathcal{P}(G\times H)$ is a chordal graph.
\end{proof}

\begin{lemma}
\label{Chordal_product_th3}
Let $G$ and $H$ be two finite non-cyclic groups such that  $|G\times H|$ has {exactly two prime divisors $p$ and $q$}. If $G$ is a $p$-group, then $\mathcal{P}(G\times H)$ is a chordal graph if and only if $G$ and $H$ have the forms described in Theorem \ref{Chordal_product_th} (iii).
\end{lemma}

\begin{proof}
First, suppose that $\mathcal{P}(G\times H)$ is a chordal graph and let $|H|=p^{r}q^{s}$, where $r, s \geq 1$. Let $Q$ be a Sylow $q$-subgroup of $H$. Here, $G\times Q$ is nilpotent and power-chordal. Now, as $G$ is non-cyclic, so $Q$ must be a cyclic group, and $G$ must be a $p$-group of exponent $p$ (by Theorem \ref{Chordal_pre_th5}). Let $P$ be a Sylow $p$-subgroup of $H$ and  $O_{p}(H)$ be the largest normal $p$-subgroup of $H$. By Bunside's theorem, $H$ is a solvable group. Now consider the Fitting subgroup $\mathrm{Fit}(H)$ of $H$.\medskip\\
\textbf{Case 1.} {$\mathrm{Fit}(H)$ be a $q$-group}.\\
{Here $\mathrm{Fit}(H)=C_{q^{k}}$ for some $k\leq s$ since the Sylow $q$-subgroup of $H$ is cyclic}. We claim that $Q$ is normal in $H$. If $Q$ is not normal in $H$, then $H$ contains non-adjacent elements, say $c$ and $d$, of order powers of $q$. On the other hand, as $G$ is non-cyclic, $G$ contains non-adjacent elements, say $a$ and $b$, of order $p$. Then $(a^q, 1)\sim (a, c)\sim (1, c^p)\sim (b, c)\sim (b^q, 1)\sim (b,d)\sim (1, d^p)\sim (a, d)\sim (a^q, 1)$ is an $8$-cycle in $\mathcal{P}(G\times H)$. This contradicts the chordality of $\mathcal{P}(G\times H)$. Therefore, either $H\cong C_{q^{s}}\rtimes C_{p^r}$ or $H\cong C_{p^r}\rtimes C_{q^s}$.\medskip\\
\textbf{Case 2.} {$\mathrm{Fit}(H)$ is a $p$-group.\\ Here $\mathrm{Fit}(H)=O_{p}(H)$.}
Then $H$ is either $O_{p}(H) \rtimes C_{q^s}$ or $C_{q^s}\rtimes O_{p}(H)$, where $O_{p}(H)$ is not necessarily a group of exponent $p$. Now, consider the group $O_{p}(H) \rtimes C_{q^s}$. We claim that $O_{p}(H)$ is a cyclic group. If  $O_{p}(H)$ is non-cyclic, then $H$ contains at least two non-adjacent elements of order $q$ or $pq$. Then as $G$ is non-cyclic, $\mathcal{P}(G\times H)$ contains a cycle of length more than $4$, which is a contradiction to the chordality of  $\mathcal{P}(G\times H)$. Thus, in this case, $H$ is either $C_{p^r}\rtimes C_{q^s}$ or $C_{q^s}\rtimes C_{p^r}$, for some $r, s\geq 1$. Also, if  $H\cong C_{p^r}\rtimes C_{q^s}$ then Sylow $q$-subgroup of $H$ must be normal, or else $\mathcal{P}(G\times H)$ contains a cycle of length more than $4$. Thus, $H\cong C_{p^{r}}\times C_{q^{s}}$. Since $H$ is nilpotent and power-chordal so according to Theorem \ref{Chordal_pre_th5} either $r=1$ or $s=1$, i.e., in this case $H$ is  either $C_{p^r}\times C_{q}$ or $C_{p}\times C_{q^s}$. 
\medskip\\
\textbf{Case 3.} {$pq$ divides $|\mathrm{Fit}(H)|$}. \\
In that case, $\mathrm{Fit}(H)=O_{p}(H)\times O_{q}(H)$, and hence $H$ is either $(O_{p}(H)\times C_{q^{k}}) \rtimes C_{q}$ with $O_{p}(H)$ being a $p$-group of exponent $p$ or $(O_{p}(H)\times C_{q^{s}})\rtimes C_{p}$. Now, as the Sylow $q$-subgroups of $H$ are cyclic, the Sylow $q$-subgroups of $\mathrm{Fit}(H)$ are also cyclic. Thus, in the second form of $H$, $O_{p}(H)$ must be a $p$-group of exponent $p$.
Also, we get two additional possibilities for $H$, namely, $C_{q}\rtimes (O_{p}(H)\times C_{q^{k}})$ or $C_{p}\rtimes (O_{p}(H)\times C_{q^s})$ where $O_{p}(H)$ are $p$-group of exponent $p$.

However, in each of the above four {possibilities}, $H$ contains non-adjacent elements, say $c_{1}$ and $c_{2}$, of order $pq$. Then, considering two non-adjacent elements $g_1$ and $g_2$ of $G$, we get an 8-cycle $(g_{1}^{q}, 1)\sim (g_1, c_{1}^{p})\sim (1, c_{1}^{p^2})\sim (g_2, c_{1}^{p})\sim (g_{2}^{q}, 1)\sim (g_2, c_{2}^{p})\sim (1, c_{2}^{p^2})\sim (g_1, c_{2}^{p})\sim (g_{1}^{q}, 1)$ in $\mathcal{P}(G\times H)$, which contradicts the chordality of $\mathcal{P}(G\times H)$. So, Case 3 can not occur.

Hence, $G$ and  $H$ have the structures as in Theorem \ref{Chordal_product_th} (iii).\medskip

We now prove the converse of the theorem. First, we assume that $G$ is a finite non-cyclic $p$-group of exponent $p$ and $H\cong (C_{q^s}\rtimes C_{p^r})$. \medskip\\
\textbf{Case I.} {$H$ is an EPPO group}.\\
{Here every element} of $H$ is either a power of $p$ or a power of $q$. Consider a cycle $\cdots x_{1}\rightarrow x_{2}\leftarrow x_{3}\rightarrow x_{4}\leftarrow x_{5}\rightarrow\cdots$ in $\overrightarrow{\mathcal{P}}(G\times H)$, where $x_{i}=(g_i, h_i)$. Suppose that $g_3=a$, $h_3=b$ with $o(a)=p$, $o(b)=q^{n}$. Without loss of generality, we choose $g_2=a^{q^n}$, $ h_2=1$ and $g_4=1$, $h_4=b^{p}$. Then $g_1\in \langle a\rangle$. Also, since the Sylow $q$-subgroup of $H$ is cyclic and normal, $h_1\sim h_3$. Hence $x_1\sim x_3$.
Again, if we take $o(a)=p$ and $o(b)=p^{s}$, then $x_{2}\sim x_{4}$. Thus, $\mathcal{P}(G\times H)$ is chordal.\medskip\\
\textbf{Case II.} {$H$ is not an EPPO group.} \\
In that case,  every non-identity element of $H$ is of order either a power of $p$ or a power of $q$ or of the form $p^{i}q^{j}$.  Consider the cycle $\cdots x_{1}\rightarrow x_{2}\leftarrow x_{3}\rightarrow x_{4}\leftarrow x_{5}\rightarrow\cdots$ in $\overrightarrow{\mathcal{P}}(G\times H)$, where $x_{i}=(g_i, h_i)$. Here arise two subcases.\medskip\\
\textbf{Subcase a.} If $g_{3}=a$, $h_{3}=b$ with $o(a)=p$, $o(b)=p^{n}$, then $x_{2}\sim x_{4}$.\medskip\\
\textbf{Subcase b.} Let $g_{3}=a$, $h_{3}=b$ with $o(a)=p$, $o(b)=q^{n}$. Without loss of generality, we can choose $g_{2}=a^{q^n}$, $h_{2}=1$, $g_{4}=1$ and $h_{4}=b^{p}$. Now $g_{1}\in \langle a\rangle$, and $h_{1}$ must adjacent to $h_{3}$ as $o(h_{1})=q$. Since the Sylow $q$-subgroup of $H$ is cyclic and normal so  $x_{1}\sim x_{3}$.\\
 {Again, if $o(a)=p$, $o(b)=p^{m}q^{n}$, then in a similar manner, just as in the above paragraph, we obtain $x_{1}\sim x_{3}$. Therefore, $\mathcal{P}(G\times H)$ is a chordal graph.}
\medskip 

Finally, by using a similar argument, we can show that $\mathcal{P}(G\times H)$ is chordal when $H \cong C_{p^r}\times C_{q}$ or $C_{p}\times C_{q^s}$. This completes the proof.
\end{proof}

\begin{lemma}
\label{Chordal_product_th4}
Let $G$ and $H$ be two finite non-cyclic groups such that $|G\times H|$ has exactly two distinct prime divisors  $p$ and $q$, and $pq$ divides both $|G|$ and $|H|$. Then $\mathcal{P}(G\times H)$ is chordal if and only if $G$ and $H$ satisfy condition (iv) of Theorem \ref{Chordal_product_th}.
\end{lemma}

\begin{proof}
First, suppose that $\mathcal{P}(G\times H)$ is a chordal graph. If both $G$ and $H$ are non-EPPO groups, then $G\times H$ contains the non-power-chordal subgroup $C_{pq}\times C_{pq}$. So, at least one of $G$ and $H$ is an EPPO group. {We assume that $G$ is an EPPO group}.\\
 Let $P_{G}$, $P_{H}$ and $Q_{G}$, $Q_{H}$ be Sylow $p$- and $q$-subgroups of $G$ and $H$, respectively. {We treat the cases where $H$ is and is not an EPPO group separately below}.\medskip\\
\textbf{Case 1.} {$H$ is a non-EPPO group}. \\
{In that case, $H$ contains an element,} say $c$, of order $pq$. Now both $P_{G}\times \langle c\rangle$ and $Q_{G}\times \langle c \rangle$ are nilpotent and power-chordal, so (by Theorem \ref{Chordal_pre_th5}) $P_{G}$ is either $C_{p}$ or a non-cyclic $p$-group of exponent $p$ (say, $P$) and $Q_{G}$ is either $C_{q}$ or a non-cyclic $q$-group of exponent $q$ (say, $Q$). Now, we observe that both $P_{G}$ and $Q_{G}$ cannot be $P$ and $Q$ respectively. {Indeed, if such were the case, then the nilpotent power chordal subgroups $P_{G}\times Q_{H}$ and $Q_{G}\times P_{H}$ are both cyclic by Theorem \ref{Chordal_pre_th5}}. Thus, we obtain that both $P_{H}$ and $ Q_{H}$ are cyclic. Also, one can observe that $P_{H}, Q_{H}$ must be normal; {otherwise} there exists a {cycle of length} more than $4$. This implies that $H$ is a cyclic group. It contradicts the assumption that $H$ is non-cyclic. {Hence, we get the following three possibilities for $P_{G}$ and $Q_{G}$:
\begin{itemize}
\item[(x)]
$P_{G}\cong C_{p}$ and $Q_{G}\cong Q$.
\item[(y)]
$P_{G}\cong P$ and $ Q_{G}\cong C_{q}$.
\item[(z)]
$P_{G}\cong C_{p}$ and $Q_{G}\cong C_{q}$.
\end{itemize}}
First, suppose that $P_{G}\cong C_{p}$ and $Q_{G}\cong Q$. Since $P_{H}\times Q_{G}$ is nilpotent (as it is the direct product of two nilpotent groups) and power-chordal, $P_{H}$ must be cyclic as well as normal. Otherwise, there exists a cycle of length greater than $4$. Now, consider the Fitting subgroup of $H$. Clearly, $\mathrm{Fit}(H)\cong C_{p}\times C_{q}$ (as $H$ is a non-EPPO group). This implies $H \cong C_{pq}\rtimes C_{q^{s-1}}$, where $s\geq 2$. Additionally, from the {structure} of $H$ it is clear that Sylow $q$-subgroup $Q_{H}$ of $H$ is non-cyclic. Then, as $P_{G}\times Q_{H}$ is nilpotent and power-chordal, $P_{G}$ is cyclic as well as normal. This gives $G\cong C_{P}\rtimes Q$.

{Now, by using a similar approach where $P_{G}$ and $Q_{G}$ are of the form (y) or (z), we obtain the structures of $G$ and $H$ as: $G\cong C_{q}\rtimes P,\ H\cong C_{pq}\rtimes C_{p^{r-1}}$, or $G\cong C_{p}\rtimes C_{q},\ H\cong C_{pq}\rtimes C_{p^{r-1}}$ or $G\cong C_{q}\rtimes C_{p}, H\cong C_{pq}\rtimes C_{q^{s-1}}$ (where $r, s \geq 2$).}
\medskip\\
\textbf{Case 2:} {$H$ is an EPPO group}.\\
Let $G$ contain an element of order $p^{k}$, where $k \geq 1$. Then $H$ can only contain elements of order $q$ but not a higher power of $q$. Moreover, if $G$ contains an element of order a power of $q$, then $H$ must be the EPPO group in which every element is of order either $p$ or $q$, {i.e.,} $H$ is an EPO group. Otherwise, both $G$ and $H$ must be EPPO groups whose every non-identity element is of order either $q$ or a power of $p$. Next, we observe that {either the Sylow $p$-subgroups or the Sylow $q$-subgroups} of both $G$ and $H$ are necessarily cyclic and normal, or else there exist two non-adjacent elements in $G$ of order $p$ and two non-adjacent elements in $H$ of order $q$ such that $\mathcal{P}(G\times H)$ contains an $8$-cycle. So both $G$ and $H$ contain either the Sylow $p$- or $q$-subgroups (but not both), which are cyclic as well as normal. Therefore, $G$ and $H$ are the groups of the form (iv) of Theorem \ref{Chordal_product_th}. \medskip

For the converse part, first we assume that $G$ and $H$ are the groups as described in (iv)(1)(a) of Theorem \ref{Chordal_product_th}. Let $\cdots x_{1}\rightarrow x_{2}\leftarrow x_{3}\rightarrow x_{4}\leftarrow x_{5}\rightarrow\cdots$ be a cycle in the directed power graph of $G\times H$, where $x_{i}=(g_i, h_i)$. Let $g_{3}=a$, $h_{3}=b$ with $o(a)=p$, $o(b)=q$. Without loss of generality, we can choose $g_{2}=a^q, h_2=1$ and $g_4=1, h_4=b^p$. So, if we add $x_5$, then $h_5\in \langle b\rangle$ and $g_5$ must be adjacent to $a$, and hence $x_3\sim x_5$. Also, if we check for $4$-cycles, then we must have $x_1\sim x_3$. Further, if we take $g_3=a, h_3=b$ with $o(a)=p^{k}$, $o(b)=p$ or $o(a)=q^{s}$, $o(b)=q$, then $x_2$, $x_3$ and $x_{4}$ form a triangle. Thus, in this case, $\mathcal{P}(G\times H)$ is a chordal graph.
 
Now, we consider the structures as described in (iv)(1)(b) of Theorem \ref{Chordal_product_th}. Let $\cdots x_{1}\rightarrow x_{2}\leftarrow x_{3}\rightarrow x_{4}\leftarrow x_{5}\rightarrow\cdots$ be a cycle in the directed power graph of $G\times H$, where $x_{i}=(g_i, h_i)$. Let $g_{3}=a$, $h_{3}=b$. If both $a$ and $b$ are of order a power of $p$ or $q$, then $x_2\sim x_4$. So, we assume that $o(a)=p$ and $o(b)=q$, and Sylow $p$-subgroups of $G$ and $H$ are cyclic and normal. Without loss of generality, we take $g_{2}=a^q, h_2=1$ and $g_4=1, h_4=b^p$. Then $x_5$ must be adjacent to $x_3$. Also, if we check for $4$-cycles, then again we have $x_1\sim x_3$. Hence, in this case, $\mathcal{P}(G\times H)$ is chordal.\medskip

Next, we consider the groups described in (iv)(2) of Theorem \ref{Chordal_product_th}. First, we assume that $G\cong C_{p}\rtimes Q$ and $H\cong C_{pq}\rtimes C_{q^{s-1}}$ (where $s\geq 2$). If possible, let $\cdots x_{1}\rightarrow x_{2}\leftarrow x_{3}\rightarrow x_{4}\leftarrow x_{5}\rightarrow\cdots$ be a cycle in the directed power graph of $G\times H$, where $x_{i}=(g_i, h_i)$. Let $g_{3}=a$, $h_{3}=b$. If $o(a)$ and $ o(b)$ are powers of the same prime, we have then $x_{2}\sim x_{4}$. Let $g_{3}=a$, $h_{3}=b^{p}$, where $o(a)=p, o(b)=pq$. Preserving the generality, we let $o(g_{2})=p, o(h_2)=1$ and $o(g_4)=1, o(h_4)=q$. It gives $h_5 \in \langle b^{p}\rangle$ and $o(g_5)=p$. So, $x_{3}\sim x_{5}$. Similarly, if we check the existence of a $4$-cycle, we have $x_{1}\sim x_{3}$. For the other cases, {i.e.,} if $o(a)=p$ or $q$, and $o(b)=q^{k}$ or $p$ or $pq$, just as before, we obtain that either $x_{3}\sim x_{5}$ or $x_{1}\sim x_{3}$. Hence, $\mathcal{P}(G\times H)$ is a chordal graph.  

For the groups in Theorem \ref{Chordal_product_th} (iv)(2)(y), (iv)(2)(z), we can show that $\mathcal{P}(G\times H)$ is a chordal graph (by the similar {procedure} as done for the groups in the case (2)(x)).  This completes the proof of the theorem.
\end{proof}

\subsection{Direct product of $G$ and $H$, where $|G\times H|$ has three or more distinct prime divisors}
In this section, we characterize finite groups $G$ and $H$ such that $|G\times H|$ has more than two distinct prime divisors, and their direct product is power-chordal. The following theorem is the main theorem of this section.
\begin{theorem}
\label{Chordal_product_th7}
Let $G$ and $H$ be two finite groups such that $|G\times H|$ has three or more distinct prime divisors. Then $\mathcal{P}(G\times H)$ is a chordal graph if and only if one of the following holds for $G$ and $H$.
\begin{itemize}
    \item[(i)] One of $G$ and $H$ is a cyclic $p$-group {for some prime $p$} and the other one is either of the following:
    \begin{enumerate}
\item 
 An EPPO group.
\item 
{A group (of order divisible by three or more distinct primes) which contains elements $h$ and  $h'$ of respective orders $pq$ and $pr$, where $p\neq q$ and $q, r$ are primes (not necessarily distinct), such that one of (x) or (y) holds.}\\
     (x) {$q\neq r$ implies $h^{q}\neq {h'}^{r}$.}\\
(y) If {$q=r$ then $h^{p}\notin \langle {h'}^{p}\rangle$.}
 \end{enumerate}
    \item[(ii)] Both $G$ and $H$ are non-cyclic EPPO groups, and none of $G$ or $H$ is a group of prime power order, and one of the following holds for $G$ and $H$.
\begin{enumerate}
    \item Every Sylow subgroups of  at least one of $G$ and $H$ are cyclic.
    \item There exists a prime, say $p$, that divides both $|G|$ and $|H|$ such that the Sylow $p$-subgroups of both $G$ and $H$ are non-cyclic $p$-group of exponent $p$ and all other Sylow subgroups of both $G$ and $H$ are cyclic as well as normal.
    \end{enumerate}
\end{itemize}
\end{theorem}
We prove Theorem \ref{Chordal_product_th7} by proving the following three lemmas.
\begin{lemma}
\label{Chordal_product_th6}
Let $G$ and $H$ be two finite non-cyclic groups such that $|G\times H|$ has three or more distinct prime divisors. If one of $G$ and  $H$ is of prime power order then $\mathcal{P}(G \times H)$ is not a chordal graph.
\end{lemma}

\begin{proof}
Let $p$, $q$, and $r$ be three distinct primes that divide $|G\times H|$. {Suppose that $|G|=p^{\alpha}$, $\alpha>1$}. Then $q$ and $r$ divide $|H|$. Since $G$ is non-cyclic, it contains at least two non-adjacent elements say $a$ and $b$. Let $c$ and $d$ be two elements of $H$ such that $o(c)=q$ and $o(d)=r$. Then $(a,c)\sim(a,1)\sim(a,d)\sim(1,d)\sim(b,d)\sim(b,1)\sim(b,c)\sim(1,c)\sim(a,c)$ is a cycle in $\mathcal{P}(G\times H)$. Therefore, $\mathcal{P}(G\times H)$ is not a chordal graph.
\end{proof}

\begin{lemma}
Let $G$ and $H$ be two finite groups such that $|G\times H|$ has three or more distinct prime divisors. If $G$ is a cyclic group of order power of some prime $p$, then $\mathcal{P}(G\times H)$ is a chordal graph if and only if $H$ is one of the following groups.
    \begin{enumerate}
\item 
 {An} EPPO group.
\item 
 {A group (of order divisible by three or more distinct primes) which contains elements $h$ and  $h'$ of respective orders $pq$ and $pr$, where $p\neq q$ and $q, r$ are primes (not necessarily distinct), such that one of (x) or (y) holds.}
 \begin{itemize}
     \item[(x)] {$q\neq r$ implies $h^{q}\neq {h'}^{r}$.}
\item[(y)] {If} {$q=r$ then $h^{p}\notin \langle {h'}^{p}\rangle$.}
 \end{itemize}

\end{enumerate}
\end{lemma}
\begin{proof}
Let $G\cong C_{p^\alpha}$, where $\alpha\geq 1$. Let $p_i$ ($i\geq 1$) be distinct primes other than $p$ that divide $|H|$, and $P_{i}$ be corresponding Sylow $p_{i}$-subgroups of $H$. Now, $G\times P_{i}$ is nilpotent and power-chordal for each $i$. If $\alpha >1$, then $P_{i}$'s are $p_{i}$-groups of exponent $p_i$ (by Theorem \ref{Chordal_pre_th5}). Also, if $\alpha =1$ then each $P_i$ is either cyclic or a group of prime exponent. Further, we observe that $H$ cannot contain any element of order $p_{i}p_{j}$ for $i\neq j$; or else $G\times H$ contains the non-power-chordal subgroup $C_{pp_ip_j}$.
Thus, either $H$ is an EPPO group or $H$ has elements of order of the form $pp_{i}$.

{Next, suppose that $q$ and $r$ are primes (not necessarily distinct) such that $H$ contains elements of orders $pq$ and $pr$. Let $q\neq r$ and $h, h'\in H$ such that $o(h)=pq$ and $o(h')=pr$. If $h^{q}={h'}^{r}$, then $\mathcal{P}(G\times H)$ contains the cycle $(a^{q}, 1)\sim (a, h^p)\sim (1, h^{p^2})\sim (1, h)\sim (1, h^{q})\sim (1, h')\sim (1, {h'}^{p^2})\sim (a, {h'}^{p})\sim (a^{q}, 1)$; and similarly if $q=r$ and $h^{p}\in \langle {h'}^{p}\rangle$, then $\mathcal{P}(G\times H)$ contains the cycle $(g^{q}, 1)\sim (g, h^p)\sim (1, h^{p^2})\sim (1, {h'}^{p^2})\sim (g, {h'}^{p})\sim (g^{q}, 1)$. Thus, $H$ is either an EPPO group or it satisfies condition (b)}.\medskip

Conversely, let $\cdots x_{1}\rightarrow x_{2}\leftarrow x_{3}\rightarrow x_{4}\leftarrow x_{5}\rightarrow\cdots$ be a cycle in $\overrightarrow{\mathcal{P}}(G\times H)$.

First, let $G\cong C_{p^{\alpha}}$ with $\alpha \geq 1$ and $H$ be an EPPO group. If $\alpha >1$, then every non-identity element of $H$ is of order either a power of $p$ or a prime other than $p$. Again, if $\alpha =1$, then each non-identity element of $H$ is of prime power order. Let $x_{i}=(g_i, h_i)$. Consider $g_3=g$ and $h_3=h$ where $o(g)=p$ and $o(h)=p^k$. Clearly, $x_2 \sim x_{4}$. If $o(g)=p$ and $o(h)=p_{i}$, then $g_2=g^{p_i}, h_2=1$ and $g_4=1, h_4=h^{p}$. So $h_{5}\in \langle h\rangle$. Since $G$ is cyclic, so $g_5\in \langle a \rangle$, and hence $x_5\sim x_3$. {Again, if we verify the existence of $4$-cycles, in a similar manner, we obtain that $x_1\sim x_3$}. Thus, $\mathcal{P}(G\times H)$ is chordal.

Next, suppose that $H$ satisfies (b). {Let $q\neq r$ and $h, h'$ be two elements of respective orders in $H$. Let $g_{3}=g, h_{3}=h^p$ such that $o(g)=p, o(h)=pq$. Then, without loss of generality, we choose $g_2=g^{q}, h_2=1$ and $g_4=1, h_4=h^{p^2}$}. Now, preserving the generality, we take $g_5=1, h_5=h$, $g_6=1, h_6=h^{q}$. {Then $g_7\in \langle g \rangle$ and $h_7 \in \langle h \rangle$ (as $G$ is cyclic)}. Thus, $x_{7}\sim x_{3}$. {Similarly, for the cycles of lengths $4$ and $6$, we obtain} either $x_{1}\sim x_{3}$ or $x_3\sim x_{6}$. Therefore, in this case, $\mathcal{P}(G\times H)$ does not contain any induced cycle of length $\geq 4$.

Again, if $q=r$ with $h^{p}\notin \langle {h'}^{p}\rangle$, $h^{p}\nsim {h'}^{p}$, then there does not exist any cycle of length $\geq 4$. Hence, in this case also, $\mathcal{P}(G\times H)$ is a chordal graph.
This completes the proof of the lemma.
\end{proof}
\begin{lemma}
\label{Chordal_product_th_x}
Let $G$ and $H$ be two finite non-cyclic groups with $|G\times H|$ has three or more distinct prime divisors, and none of $G$ and $H$ is a group of prime power order. Then $\mathcal{P}(G \times H)$ is chordal if and only if $G$ and $H$ are EPPO groups satisfying one of the following.
\begin{enumerate}
    \item Each Sylow subgroup of at least one of $G$ or $H$ is cyclic.
    \item There exists a prime, say $p$, that divides both $|G|$ and $|H|$ such that the Sylow $p$-subgroups of both $G$ and $H$ are non-cyclic $p$-group of exponent $p$ and all other Sylow subgroups of both $G$ and $H$ are cyclic as well as normal.
\end{enumerate}
\end{lemma}

\begin{proof}
Let $G$ and $H$ be two finite non-cyclic groups {such that $|G\times H|$ have three or more distinct prime divisors, }say $p$, $q$ and $r$. Suppose that $\mathcal{P}(G\times H)$ is chordal. We observe that both $G$ and $H$ must be EPPO groups, or else $G\times H$ contains a non-power-chordal subgroup. 
{Now, consider two cases according to number of common prime divisors of $|G|$ and $|H|$.}\medskip\\
\textbf{Case 1.} {$|G|$ and $|H|$ have no common prime divisor.} \\
Let $p$ be a prime that divides $|G|$ with the Sylow $p$-subgroups of $G$ are non-cyclic. Then $H$ satisfies (a) by Theorem \ref{Chordal_pre_th5} (since for any prime say $p_{i}$ divides $|H|$, $P\times P_{i}$ is nilpotent and power-chordal; here $P$, $P_{i}$ are Sylow $p$-subgroup of $G$ and Sylow $p_{i}$-subgroup of $H$ respectively). {Again if the Sylow $p$-subgroups of $G$ are cyclic, then either $G$ or $H$ satisfy the condition (a).}\medskip\\
\textbf{Case 2.} {$|G|$ and $|H|$ has at least one common prime divisor.}\\
Let $p$ be a prime that divides both $|G|$ and $|H|$, and the Sylow $p$-subgroup of $G$ (say, $P_{G}$) be non-cyclic. {For the prime divisors $p_{i}(\neq p)$ of $|H|$, let $P_{i}$ be the Sylow $p_{i}$-subgroups of $H$. Since $P_{G}\times P_{i}$ is nilpotent and power-chordal and $P_{G}$ is non-cyclic, by Theorem \ref{Chordal_pre_th5} each Sylow $p_{i}$-subgroup ($p_{i}\neq p$) of $H$ must be cyclic. Also, we observe that each Sylow $p_{i}$-subgroups ($p_{i}\neq p$) must be normal as well.} Otherwise, $H$ contains two non-adjacent elements of order $p_{i}\neq p$, for some $i$. Also, as Sylow $p$-subgroups of $G$ are non-cyclic, it contains two non-adjacent elements of order $p$. Then $\mathcal{P}(G\times H)$ contains a cycle of length greater than $4$. Moreover, as $P_{G}\times P_{i}$ is nilpotent and power-chordal with $P_{i}$ being cyclic and $P_{G}$ being non-cyclic, $P_{G}$ must be a {non-cyclic $p$-group of exponent $p$}. {Consider the following subcases based on Sylow p-subgroups of $H$.}\medskip\\
\textbf{Subcase I.} If the Sylow $p$-subgroups of $H$ are cyclic, then each Sylow subgroup of $H$ are cyclic; and $H$ satisfies condition (a).\medskip\\
\textbf{Subcase II.} If Sylow $p$-subgroups of $H$ is non-cyclic, then again we can conclude that all Sylow $p_{j}$-subgroups ($p_{j}\neq p$) of $G$ are cyclic and normal and Sylow $p$-subgroups of $H$ are {non-cyclic $p$-group of exponent $p$}. So, in this case, $G$ and $H$ satisfies the condition (b). 

Similarly, if we consider that Sylow $p$-subgroups of $G$ are cyclic, then either (a) or (b) hold.\medskip

For the converse of the theorem, first we assume that $H$ is the non-cyclic EPPO group such that all Sylow subgroups of $H$ are cyclic. 
We now show that $\mathcal{P}(G\times H)$ is chordal. Suppose that $\overrightarrow{\mathcal{P}}(G\times H)$ contains a cycle $\cdots \leftarrow x_{1}\rightarrow x_{2}\leftarrow x_{3}\rightarrow \cdots$, where $x_{i}=(g_i, h_i)$. Let $g_{3}=a$ and $h_{3}=b$. If both $o(a)$ and $o(b)$ are powers of the same prime, then $x_{2}\sim x_{4}$. Without loss of generality, we assume that $o(a)$ is a power of $p$ and $o(b)$ is a power of $q$ ($\neq p$). Let $o(a)=p^{m}$ and  $o(b)=q^{n}$. Then either $g_{2}=a^{q^n}, h_{2}=1$, $g_{4}=1, h_{4}=b^{p^m}$ or the reverse. Preserving the generality, suppose that $g_{2}=a^{q^n}, h_{2}=1$, $g_{4}=1, h_{4}=b^{p^m}$. Then $g_{1}\in \langle a \rangle$ and $o(h_{1})=q^{n}$. As Sylow $q$-subgroups of $H$ are cyclic, so $h_{1}\in \langle b \rangle$. Hence, $x_{1}\sim x_{3}$. Therefore, $\mathcal{P}(G\times H)$ is a chordal graph.

Next, suppose that $G$ and $H$ are two non-cyclic EPPO groups and that condition (b) holds. If possible, let $\cdots \leftarrow x_{1}\rightarrow x_{2}\leftarrow x_{3}\rightarrow \cdots$ be a cycle in $\overrightarrow{\mathcal{P}}(G\times H)$, where $x_{i}=(g_i, h_i)$. Suppose that $g_{3}=a, h_{3}=b$.
If both $a$ and $b$ are powers of the same prime, then $x_{2}\sim x_{4}$.
So we assume that $o(a)$ and $o(b)$ are powers of different primes. Without loss of generality, let $o(a)=p$ and $ o(b)=q^{s}$. Then, either $g_{2}=a^{q^s}, h_{2}=1$, $g_{4}=1, h_{4}=b^{p}$ or the reverse. Preserving the generality, we assume that $g_{2}=a^{q^s}, h_{2}=1$, $g_{4}=1, h_{4}=b^{p}$. Then again, we get $x_{1}\sim x_{3}$. Thus, $\mathcal{P}(G\times H)$ does not contain any cycle of length $4$ and above. Hence, $\mathcal{P}(G\times H)$ is a chordal graph.

Following a similar approach, for the other possibilities of $o(a)$ and $o(b)$, we conclude that $\mathcal{P}(G\times H)$ is a chordal graph.
\end{proof}


 \section{Symmetric groups}
In this section, we compute the values of $n$ for which the symmetric group $S_n$ is a power-chordal group. We observe that Theorem 4.1 of \cite{Mehatari} also holds for the chordality of the power graph of any finite group. So, we get the following theorem.
\begin{theorem}
Let $G$ be a finite group such that any two distinct maximal cyclic subgroups of $G$ intersect at the identity. Then  the  power graph of $G$ is a chordal graph  if and only if the orders of the maximal cyclic subgroups are of either a prime power or a product of some prime and a prime power.
\end{theorem}

\begin{theorem}
\label{Chordal_symm_th1}
$\mathcal{P}(S_n)$ is a chordal graph if and only if $n\leq 5$.
\end{theorem}

\begin{proof}
For $n \geq 6$, the power graph of $S_{n}$ always contains  $(a,b,c)(d,e) \sim (d,e) \sim (a,b,f)(d,e) \sim (a,b,f) \sim (a,b,f)(c,d) \sim (c,d) \sim (a,b,e)(c,d) \sim (a,b,e) \sim (a,b,e)(d,f) \sim (d,f) \sim (a,b,c)(d,f) \sim (a,b,c) \sim (a,b,c)(d,e)$, a cycle of length $12$.

If $n \leq 5$, then any two distinct maximal cyclic subgroups of $S_n$ intersect at the identity.  The orders of those maximal cyclic subgroups are, respectively, $\{2\}$ for $(n=2)$, $\{2,3\}$ for $(n=3)$, $\{2,3,4\}$ for $(n=4)$ and $\{4,5,6\}$ for $(n=5)$. The power graphs of all these cyclic subgroups are chordal.

Hence $\mathcal{P}(S_n)$ is a chordal graph if and only if $n\leq 5$..
\end{proof}


\section{Simple groups}
In this section, we try to identify finite simple groups whose power graph is a chordal graph. The cyclic group $C_{p}$ (for prime $p$) belongs to the class of finite simple groups, and its power graph is chordal by Theorem \ref{Chordal_pre_th1}. In the subsequent subsections, we discuss the chordality of the power graphs of the other finite simple groups, namely, the alternating groups, the simple groups of Lie type, and the sporadic simple groups. To draw the conclusion regarding the {chordality} of simple groups, we used information from the book \cite{Wilson} and the papers \cite{King, Suzuki, Suzuki_1, Kleidman_{1}, Kleidman, Conway, Bruce, Malle}.
\subsection{Alternating groups}
\begin{theorem}
\label{chordal_alternating}
Let $A_{n}$ be the alternating group on $n$ symbols. Then the power graph of $A_{n}$ is chordal if and only if $n \leq 7$.
\end{theorem}

\begin{proof}
If $n \geq 8$, $\mathcal{P}(A_{n})$ contains a cycle  
\begin{eqnarray*}
(a, b)(c, d) \sim (a, b)(c, d)(e, f, g) \sim (e, f, g) \sim (a, c)(b, d)(e, f, g) \sim (a, c)(b, d) 
\sim\\
(a, c)(b, d)(e, f, h) \sim (e, f, h) \sim (a, b)(c, d)(e, f, h) \sim (a, b)(c, d);\end{eqnarray*} 
so its  power graph is not a chordal graph.

The alternating group $A_{3}$ is  basically the cyclic group $C_3$. Thus it follows that the power graph of $A_3$ chordal.

For $n=4$, $5$ and $6$, the orders of the cyclic subgroups of $A_n$  are  $\{3,2\}$ for $(n=4)$, $\{5, 3, 2\}$ for $(n=5)$ and $\{5, 4, 3, 2\}$ for $(n=6)$ respectively, and the maximal cyclic subgroups intersect at the identity. As all these cyclic subgroups have complete power graphs, so $\mathcal{P}(A_n)$ is chordal for $n=4, 5,6$.\medskip

Now we consider $n=7$. Let $a\in A_7$, then $o(a)\in\{2,3,4,5,6,7\}$. {Suppose for the sake of contradiction that $\mathcal{P}(A_7)$ contains an induced cycle of length at least $4$}. Then that cycle must contains at least one element of order $6$. Let $x=(a,b,c)(d,e)(f,g)$ be such an element, and $y$ and $z$ be the neighbours of $x$. Then either $y=(d,e)(f,g)$ or $z=(d,e)(f,g)$. Suppose that $y=(d,e)(f,g)$. Now in $\mathcal{P}^*(A_7)$, the neighbours of $y$ are $x$, $x^{-1}$ or an element of order $4$; however, for a neighbour of $y$ other than $x$ we can choose none of them. Therefore, an induced cycle of {length at least $4$} is not possible. So $A_7$ is power-chordal.
This completes the proof.
\end{proof}

\subsection{ Lie type simple groups of rank $1$}
The groups  $\PSL(2,q)$, $\PSU(3,q)$, $\Sz(q)$ where $q=2^{2e+1}$, and ${}^{2}{G_{2}}(q)=R_{1}(q)$ where $q=3^{2e+1}$ are the only simple groups of Lie type of  rank $1$. In this section, we examine the chordality of the power graph of these groups.

\begin{theorem}
\label{Chordal_psl}
Let $G$ be the group $\PSL(2,q)$, where $q$ is a prime power, then the power graph of $G$ is a chordal graph if and only if 
\begin{enumerate}
    \item 
   both $q+1$ and $q-1$ are either powers of some prime or {products of a prime and a prime power, provided $q$ is even;}
    \item
   both $(q+1)/2$ and $(q-1)/2$ are either prime powers or {products of a prime and a prime power, provided $q$ is odd.}
\end{enumerate}
\end{theorem}

\begin{proof}
First, we consider $q$ to be even, i.e., $q$ is a power of $2$. Then $\mathcal{P}^*(G)$ is a disjoint union of proper power graphs of cyclic groups of orders $q+1$ and $q-1$, along with some isolated vertices. As a result, $\mathcal{P}(G)$ is a chordal graph if and only if $\mathcal{P}(C_{q+1})$ and $\mathcal{P}(C_{q-1})$ are both chordal. Therefore, by Theorem \ref{Chordal_pre_th5}, $\mathcal{P}(G)$ is chordal if and only if both $q+1$ and $q-1$ are of the form in (a).\medskip

Next, suppose that $q$ is a power of an odd prime. Then $\mathcal{P}^*(G)$ is a disjoint union of proper power graphs of $C_{(q \pm 1)/2}$ along with some isolated vertices. So again, by Theorem \ref{Chordal_pre_th5}, we conclude that, $\mathcal{P}(G)$ is chordal if and only if both $(q \pm 1)/2$  are of the form in (b).
\end{proof}

\begin{theorem}
The power graph of $\PSU(3, q)$, where $q$ is a power of $2$ and $q \geq 4$, is always a non-chordal graph.
\end{theorem}

\begin{proof}
Let $\beta$ be a generator of the multiplicative group of $GF(q^2)$. Then the order of $\beta ^{q-1}$ is $q+1$. Let $p$ be a prime such that $p$ divides $q+1 (>3)$. We set $d=(q+1)/p$. Then the order of $\alpha =\beta ^{d(q-1)}$ is $p$ and $\alpha \in GF(q^2)$.
Consider the elements 
$$h=\begin{bmatrix}
\alpha & 0 & 0 \\ 0 & \alpha & 0 \\ 0 & 0 & \alpha^{-2}
\end{bmatrix}\ 
\text{ and }
k=\begin{bmatrix}
0 & \alpha & 0 \\ 1 & 0 & 0 \\ 0 & 0 & \alpha^{-1}
\end{bmatrix}.$$
Then, $h,k\in SU(3,q)$ and $hk=kh$. Also, we have $k^2 =h$, and $o(h)=p$, $o(k)=2p$, $o(hk)=2p$. 
So, $SU(3,q)$ contains a cycle $(k^p, hk^p, h, k, k^p)$.
Now let $Z$ denote the {center} of $SU(3,q)$. Then $(k^{p}Z, hk^{p}Z, hZ, kZ, k^{p}Z)$ is cycle of length $4$ in $\mathcal{P}(\PSU(3,q))$. \medskip 

However, the above argument does not work if $q=8$. So we conclude our proof by showing $\PSU(3,8)$ is non-power-chordal. Note that,
$\PSU(3,8)$ contains a subgroup isomorphic to $C_{9}\rtimes \SL_{2}(3)$ which contains  $C_{3}\times Q_{8}$. Now, by Theorem \ref{Chordal_pre_th5}, the power graph of $C_{3}\times Q_{8}$ is not chordal, and hence $\mathcal{P}(\PSU(3,8))$ is non-chordal.
\end{proof}

\begin{theorem}
\label{PSU_odd}
Let $q=p^{n}$, where $p$ is an odd prime and $n$ be any positive integer. Then the power graph of $\PSU(3,q)$ is not a chordal graph.
\end{theorem}

\begin{proof}
If $q$ is odd, then $\PSU(3,q)$ contains a cyclic subgroup of order $(q^{2}-1)/gcd(q+1, 3)$. Since $q$ is odd, both the numbers $q-1$ and $q+1/gcd(q+1, 3)$ are even. Thus, $\mathcal{P}(\PSU(3, q))$ is chordal if $(q-1)(q+1)/gcd(q+1, 3)$ is either a power of $2$ or of the form $2^{k}p'$, where $p'$ is an odd prime. \medskip

First, suppose that both $q+1/gcd(q+1, 3)$ and $q-1$ are powers of $2$. As only one of $q+1$ and $q-1$ is divisible by $4$, the pair $(q-1,q+1)$ is either $(2,4)$ or $(4,6)$. Hence, $q=3$ or $5$. \medskip

Next, suppose that $(q-1)(q+1)/gcd(q+1, 3)=2^{k}p'$. Then one of $q-1$ and $q+1/gcd(q+1, 3)$ is a power of $2$. {First, we assume that $q-1$ is a power of $2$.} Now if $q \neq 3, 9$, then $q+1$ is either of the forms $2p'$ or $6p'$ for some odd prime $p'$. If $q+1=6p'$, then $\PSU(3, q)$ contains a non-power-chordal subgroup $C_{2^k}\times C_{2p}$. Again, if $q+1=2p'$, then $\PSU(3, q)$ contains the subgroup $C_{q+1}\times C_{q+1/gcd(q+1,3)}$. Then $C_{2p'}\times C_{2p'}$ is a non-power-chordal subgroup of $\PSU(3, q)$. So either $q=3$ or $q=9$.
\medskip

{Again, if $(q+1)/gcd(q+1, 3)$ is a power of $2$, then $q-1$ must be divisible by $2p'$ (as $q$ is odd). Then $\PSU(3, q)$ contains the non-power-chordal subgroup $C_{2^r}\times C_{2p'}$.}\medskip

We now conclude the theorem by proving the power graphs of $\PSU(3,3)$, $\PSU(3,5)$ and $\PSU(3,9)$ are non-chordal.\medskip

If $q=3$, then $\PSU(3,3)$ contains $4\cdot S_{4}$. Now, $4\cdot S_{4}$ is given by $\langle a,b,c,d,e | a^{4}=d^{3}=1, b^{2}=c^{2}=e^{2}=a^{2}, ab=ba, ac=ca, ad=da, eae^{-1}=a^{-1}, cbc^{-1}=a^{2}b, dbd^{-1}=a^{2}bc, ebe^{-1}=bc, dcd^{-1}=b, ece^{-1}=a^{2}c, ede^{-1}=d^{-1}\rangle$. Then, the power graph of $4\cdot S_{4}$ contains the cycle $(aed, ad, d, ed, e, aed)$. Therefore $\mathcal{P}(\PSU(3,3))$ is not a chordal graph.\medskip

If $q=5$, then using GAP and Sage \cite{GAP, Sage} we observe that $G$ has $525$ elements of order $2$, $3500$ elements of order $3$, $10500$ elements of order $6$, $12600$ elements of order $10$ and $15624$ elements of order $5$. Now a subgroup of order $2$ is contained in $6$ distinct cyclic subgroups of order $10$ and $10$ distinct cyclic subgroups of order $6$. Similarly, a subgroup of order $3$ is shared by $3$ distinct subgroup of order $6$. On the other hand, among the elements of order 5, 525 are contained in the cyclic subgroups of order $10$; so an element of order $5$ is contained in $25$ different cyclic subgroups of order $10$. As a result, we can choose elements $a, b$ of order $10$ and elements $c, d$ of order $6$ such that $a^{2}=b^{2}$, $c^{2}=d^{2}$, $c^{3}=b^{5}$ and $d^{3}=a^{5}$. Then $(a, a^{2}, b, b^{5}, c, c^{2}, d, d^{3}, a)$ is an $8$-cycle in $\mathcal{P}(\PSU(3,5))$ and this confirms $\PSU(3,5)$ is a non-power-chordal group.\medskip

If $q=9$, then $C_{8}\times C_{10}$ contained in $\PSU(3, q)$. So, by Theorem \ref{Chordal_pre_th5}, $C_{8}\times C_{10}$ is a non-power-chordal group.

Therefore, if $q$ is a power of an odd prime, the power graph of $\PSU(3,q)$ is not a chordal graph.
\end{proof}

\begin{theorem}
The power graph of $\Sz(q)$, where $q=2^{2e+1}$,  is a chordal graph if and only if each of the numbers $q-1$,  $q+ \sqrt{2q} +1$ and $q- \sqrt{2q}+1$ are either prime powers or a product of a prime and a prime power.
\end{theorem}

\begin{proof}
The group $\Sz(q)$ contains maximal cyclic subgroups of orders $4$, $q-1$, $q+\sqrt{2q} +1$, and $q- \sqrt{2q}+1$. So, if the power graph of $\Sz(q)$ contains a cycle of length at least $4$, then the cycle must be contained in one of the maximal cyclic subgroups. Hence, by Theorem \ref{Chordal_pre_th1} and Theorem \ref{Chordal_pre_th5}, these three numbers are either a power of some prime or of the form $p_{1}^{a}p_{2}$ where $p_1$ and $p_2$ are distinct primes.
The converse of the theorem is obvious.
\end{proof}

\begin{theorem}
Let $G$ be the Ree group ${}^{2}{G_{2}}(q)=R_{1}(q)$ where $q=3^{2e+1}$. Then there exist infinitely many values of $q$ such that $\mathcal{P}(G)$ is  a chordal graph.
\end{theorem}

\begin{proof}
We observe that $C_{2} \times \PSL(2,q)$ is the centralizer of an involution in the group $G$. The group $C_{2}\times \PSL(2,q)$ carries the subgroups $C_2 \times C_{(q \pm 1)/2}$. Therefore, by Theorem \ref{Chordal_pre_th5}, the power graph of $G$ is a chordal graph if $(q \pm 1)/2$ is either a power of $2$ or a power of an odd prime or of the form $2p^{r}$.

{Let both $q\pm 1$ be powers of $2$. Choose $q-1=n$. Then $q+1=n+2$. If both $n$ and $n+2$ are powers of $2$ then $n=2$, and this gives $q=3$. But ${}^{2}{G_{2}}(3)=R_{1}(3)$ is not simple, although this group is a non-power-chordal group (we checked this using GAP \cite{GAP}).}


Let $q+1=2p_1^{r_1}$ and $q-1=2p_2^{r_2}$, where $r_1$ and $r_2$ are odd primes. Then the diophantine equation $p_1^{r_1}-p_{2}^{r_2}=1$ has a solution. This leads to a contradiction to Theorem \ref{Chordal_pre_th8} as both $p_1$ and $p_2$ are odd.

Similarly, both $q\pm 1$ can not be of the form $4p^{k}$ as the diophantine equation $2x-2y=1$ has no solution.

Again, if any one of $q+1$ or $q-1$ is $4p_1^{s}$ and the other one is $2p_{2}^{t}$. Then, for $x=p_1^{s}$ and $y=p_{2}^{t}$, the corresponding diophantine equation is either $2x-y=1$ or $x-2y=1$. One can check that solutions exists for infinitely many values of $x$ and $y$.
So in this case, there exist infinitely many values of $q$ for which $G$ is a power-chordal group.
\end{proof}

\subsection{Lie type simple groups of rank $2$}
The groups $\PSL(3,q)$, $\PSp(4,q)$, $\PSU(4,q)$, $\PSU(5,q)$, $G_{2}(q)$, ${}^{2}{F_{4}}(q)$, and ${}^{3}{D_{4}}(q)$ belong to the class of simple groups of Lie type of rank $2$. In this section, we prove that there are only two groups, namely $\PSL(3,2)$ and $\PSL(3,4)$, whose power graph is a chordal graph.
\begin{theorem}
If $q$ is a power of $2$, then the power graph of $\PSL(3,q)$ is a chordal graph if $q=2, 4$.
\end{theorem}

\begin{proof}
First, suppose that $q=2^n$ where $n >2$. If $q$ is an odd power of $2$, then $3$ does not divide $q-1$; on the other hand, if $q$ is an even power of $2$, then $q-1$ is not a power of 3 (according to the solution of Catalan's conjecture). Hence, $q-1$ has a prime divisor greater than 3.

Let $\alpha$ be an {element of} the multiplicative group of $GF(q)$ of order $p>3$, where $p$ is a prime. Now, consider the elements 
$$h=\begin{bmatrix}
\alpha & 0 & 0 \\ 0 & \alpha & 0 \\ 0 & 0 & \alpha^{-2}
\end{bmatrix}, \
k=\begin{bmatrix}
0 & \alpha & 0 \\ 1 & 0 & 0 \\ 0 & 0 & \alpha^{-1}
\end{bmatrix}.$$

Then, we get $o(h)=p$, $o(k)=2p$, $hk=kh$ and $k^{2}=h$. Therefore, $\SL(3,q)$ contains the cycle $(k^{p}, hk^{p}, h, k, k^{p})$. Since neither $h$ nor $k$ contains a non-identity scalar matrix, $\SL(3,q)/Z(\SL(3,q))$ also contains a cycle of length $4$.\medskip

We now show that the groups $\PSL(3,2)$ and $\PSL(3,4)$ are power-chordal. Since $\PSL(3,2)$ is isomorphic to $ \PSL(2,7)$, the power graph of $\PSL(3,2)$ is chordal by Theorem \ref{Chordal_psl}. On the other hand $\PSL(3,4)$ is an EPPO group; hence, by Corollary \ref{Chordal_pre_cor4}, the power graph of $\PSL(3,4)$ is a chordal graph. This completes the proof of the theorem.
\end{proof}

\begin{theorem}
If $q$ is a power of an odd prime $p$. Then the power graph of $\PSL(3,q)$ is not a chordal graph.
\end{theorem}

\begin{proof}
Since $p$ is odd, $q-1$ is even. We now consider the following cases  \medskip

\textbf{Case I:} Let $q-1=2^{k}\geq4$. Consider the elements

$$g=\begin{bmatrix}
\alpha & 0 & 0 \\ 0 & \alpha & 0 \\ 0 & 0 & \alpha^{-2}
\end{bmatrix}\ 
\text{ and }\ 
h=\begin{bmatrix}
0 & -1 & 0 \\ 1 & 1 & 0 \\ 0 & 0 & 1
\end{bmatrix},$$
where $\alpha$ is an element of order $4$ in the multiplicative group $GF(q)$. Then $o(gh)=o(gh^2)$, $gh=hg$, and $gh^2=h^{2}g$. Then, the power graph of $\SL(3,q)$ contains a cycle $(g^{2}, gh, h^2, gh^{2}, g^{2})$. Now, as none of $g$ and $h$ is a non-identity scalar matrix, it follows that poer graph of $\PSL(3,q)$ contains a cycle of length $4$. \medskip

\textbf{Case II:} Let $q-1$ have an odd prime divisor, say $p'$, and let $\alpha$ be an element of order $p'$ in the multiplicative group $GF(q)$. Let 
$$g=\begin{bmatrix}
\alpha &0 & 0 \\ 0 & \alpha & 0 \\ 0 & 0 & \alpha^{-2}
\end{bmatrix}
\text{ and }
h=\begin{bmatrix}
0 & \alpha & 0 \\ 1 & 0 & 0 \\ 0 & 0 & \alpha^{-1}
\end{bmatrix}.$$
Then, the power graph of $\SL(3,q)$ contains a 4-cycle $(k^p, hk^{p}, h, k, k^p)$ and hence power graph of $\PSL(3,q)$ contains a $4$ length cycle. \medskip

Now, to complete the proof of the theorem, we must show that $\PSL(3,3)$ is non-power-chordal group. Note that $\PSL(3,3)$ contains a subgroup $3^{2}:2S_{4}$; and the group $3^{2}:2S_{4}$ contains the dihedral group $D_{36}$ whose power graph is not a chordal graph. Hence, the power graph of $\PSL(3,3)$ is not a chordal graph.
\end{proof}
In order to investigate the chordality of the power graph of $\PSp(4,q)$, we will now review a theorem.
\begin{theorem}[Mitchell Theorem, \cite{King}]
\label{Chordal_pre_th9}
Assuming $p>2$, the maximal subgroups of $\PSp(4,q)$ {are as follows:}
\begin{enumerate}
    \item 
    $\PSp(4,q_{0})$, where $q$ is an odd prime power of $q_{0}$;
    \item 
    $PGSp(4, q_{0})$ where $q_{0}^{2}=q$;
    \item 
    the stabilizer of a point and a plane, having index $q^{3}+q^{2}+q+1$;
    \item 
    the stabilizer of a parabolic congruence, having index $q^{3}+q^{2}+q+1$;
    \item 
     the stabilizer of a hyperbolic congruence,having index $q^{2}(q^{2}+1)/2$;
    \item 
    the stabilizer of a elliptic congruence, having index $q^{2}(q^{2}-1)/2$;
    \item 
    the stabilizer of a quadratic, having index $q^{3}(q^{2}+1)(q+1)/2$, provided $q>3$;
    \item 
    the stabilizer of a quadratic, having index $q^{3}(q^{2}+1)(q-1)/2$, provided $q>3$;
    \item 
    the stabilizer of a twisted cubic,  having index $q^{3}(q^{4}-1)/2$, for $q>7, p>3$;
    \item 
    groups of order $1920$ (for prime $q$ and $q\equiv \pm 1$(mod $8$), $960$ (for prime $q$ and $q\equiv \pm 3$(mod $8$)), $720$ (for $q$ prime and $\equiv \pm 1$(mod $12$)), $360$ (for $q \neq 7$ is a prime and $\equiv \pm 5$(mod $12$)), $2520$ (for $q=7$) or $A_{7}$.
\end{enumerate}
\end{theorem}

\begin{theorem}
The power graph of  $\PSp(4,q)$ is not a chordal graph.
\end{theorem}
\begin{proof}
Let $G=\PSp(4,q)$. First, suppose that $q$ is a power of 2. Then $\PSp(4,q)$ contains $\PSL(2,q)\times \PSL(2,q)$, and so it contains $C_{q\pm 1}\times C_{q\pm 1}$. Now $(q+1, q-1)=1$ and $3$ divides one of them. Thus, $\mathcal{P}(G)$ is a chordal graph if and only if one of $q\pm1$ is a prime and the other is a power of another prime. So we must have $q=2$, $4$ or $8$. \medskip

Next, suppose that $q$ is a power of an odd prime. Then $G$ contains the central product of two copies of $\SL(2,q)$, and hence it contains $C_{q\pm 1}\times C_{(q\pm 1)/2}$. Thus if $\PSp(4,q)$ is power-chordal, then both $(q\pm 1)/2$ have to be prime powers (as $C_{q\pm 1}\times C_{(q\pm 1)/2}$ contained in $\PSp(4, q)$).
Now one of $(q-1)/2$ or $(q+1)/2$ is even, so one of $(q\pm 1)/2$ must be a power of $2$. This implies one of $C_{q\pm 1}$ must be $4$, or else $\PSp(4,q)$ contains a non-power-chordal subgroup.
 Thus the possible values of $q$ are $3$ and  $5$.   \medskip

If $q=2$, then $\PSp(4,2)$ is isomorphic to the non-power-chordal group $S_6$. If $q=3$ or $4$, then $\PSp(4,q)$ contains $S_6$, and so its power graph is non-chordal.
Again, $\PSp(4,5)$ contains the subgroup $S_{3} \times S_{5}$ whose power graph is not a chordal graph.
Finally, by Theorem \ref{Chordal_pre_th9}, the group $\PSp(4,8)$ contains $\PSp(4,2)$, and so $\mathcal{P}(\PSp(4,8))$ not a chordal graph.
Thus, $\PSp(4,q)$ is a non-power-chordal group.
\end{proof}

\begin{theorem}
\label{G_{2}(q)}
The power graph of $G_{2}(q)$ is not a chordal graph.
\end{theorem}

\begin{proof}
We observe that both $\SL(3,q)$ and $\SU(3,q)$ are the subgroups of $G_{2}(q)$ (see \cite{Bruce,Kleidman}). Now $\SL(3, q)\cong \PSL(3, q)$ if $q \not\equiv 1(mod\ 3)$ and $\SU(3, q)\cong \PSU(3,\ q)$ if $q \not\equiv -1(mod\ 3)$. Thus, for any $q$, either $\PSL(3,q)$ or $\PSU(3,q)$ is contained in $G_{2}(q)$. Now, $\PSL(3,q)$ is power-chordal if $q=2, 4$, whereas $\PSU(3,q)$ is non-power-chordal for every $q\neq2$. So we only need to check for $q=2$. The group $G_{2}(2)$ is not simple, and its power graph is not a chordal graph since it contains $\PSU(3,3)$. Hence, the power graph of $G_{2}(q)$ is not a chordal graph.
\end{proof}

We now check the chordality of the power graph of other simple groups of Lie type of rank $2$.
\begin{itemize}
\item Let $G=\PSU(4,q)$. If $q>2$, then $G$ contains the non-power-chordal group $\PSp(4,q)$. Again, if $q=2$, then $G$ isomorphic to $\PSp(4,3)$. Hence, the power graph of $\PSU(4,q)$ is not chordal.
\item 
{Consider the group $\PSU(5,q)$}. If $q=2$, then $\PSU(5, q)$ contains a non-power-chordal subgroup $\PSU(4, 2)$. Again, if $q>2$, then $\PSU(5, q)$ contains $\SL(2, q)\times \SL(2, q)$ (see \cite{Wilson}). Now, by Theorem \ref{Chordal_product_th1}, the power graph of $\SL(2, q)\times \SL(2, q)$ is a non-chordal graph, and so $\mathcal{P}(\PSU(5, q))$ is a non-chordal graph.
\item The group ${}^{2}{F_{4}}(2^d)$ contains ${}^{2}{F_{4}}(2)$ for all odd $d$ (see \cite{Malle}), and so it contains $\PSL(3,3)$. Thus, ${}^{2}{F_{4}}(2^d)$ is non-power-chordal.
\item The group ${}^{3}{D_{4}}(q)$ contains $G_2(q)$ (see \cite{Kleidman_{1}}); so its power graph is also a non-chordal graph.
\end{itemize}
\subsection{Lie type simple groups of rank more than $2$}
Let $G$ be a simple group of Lie type of rank greater than $2$. Now, as the Dynkin diagram of $G$ has a single bond in each case,  $G$ contains $\PSL(3,q)$ as a subgroup. But the only values for which the power graph of $\PSL(3,q)$ is chordal are $q = 2, 4$. Hence, we have to check only for those simple groups whose underlying fields are finite fields with $2$ and $4$ elements, respectively.

Now $\PSL(4,2)$ is isomorphic to $A_8$. So its power graph is a non-chordal graph. The group $\PSL(4,4)$ contains $\PSL(4,2)$, whose power graph is not a chordal graph. Therefore, the power graph of $\PSL(4,4)$ {is never chordal.}

{Observe that} $\PSp(6,2)$ contains $S_8$, whose power graph is a non-chordal graph. Again, $\PSp(6,4)$ contains $\PSp(6,2)$. So, in this case, we get both $\PSp(6, 2)$, $ \PSp(6, 4)$ are non-power-chordal groups.

If $q=2,4$, then the orthogonal and unitary groups of Lie type of rank $3$ contains $\PSp(4,q)$. Hence, they are non-power-chordal groups.

\subsection{Sporadic groups}
There are $26$ sporadic groups, namely, 5 Mathieu groups ($M_{11}, M_{12}, M_{22}, M_{23}, M_{24}$), 3 Conway groups ($Co_1,Co_2,Co_3$), 3 Fischer groups ($Fi_{22},Fi_{23},Fi_{24}$), 4 Janko  groups ($J_1,J_2,J_3,J_4$),  the Monstar group ($M$), the Baby Monster group ($B$), the McLaughlin  group ($M^cL$), the Higman-Sims group ($HS$), the Held group ($He$), the Rudvalis group ($Ru$), the suzuki group ($Suz$), the Harada-Norton group ($HN$), the O'Nan group ($O'N$), the Lyons group ($Ly$) and the Thompson group ($Th$).  In this section, we are going to show that none of the above groups are power-chordal.

\begin{theorem}
\label{chordal_sps_th1}
{$M_{11}$ is not power-chordal.}
\end{theorem}
\begin{proof} We observe that the Mathieu group $M_{11}$ has elements of order 2, 3, 4, 5, 8, and 11 (see \cite{Conway, Kumar}); in which the elements of order $5$ and $11$ form disjoint complete graphs in $\mathcal{P}^*(M_{11})$. Apart from these elements, there are $165$ elements of order $2$, $440$ elements of order $3$, $990$ elements of order $4$, $1320$ elements of order $6$, and $1980$ elements of order $8$.
Then we can choose $4$ elements, say $a$, $b$, $c$, and $d$, each of order $6$ such that $a^{2}=b^{2}$, $a^{3}=d^{3}$, $b^{3}=c^{3}$, $c^{2}=d^{2}$. Then $\mathcal{P}^*(M_{11})$ contains an $8$-cycle $(a, a^{2},b, b^{3},c, c^{2}, d, d^{3}, a)$. Hence, $\mathcal{P}(M_{11})$ is a non-chordal graph.
\end{proof}

\begin{theorem}
\label{chordal_sps_th2}
{$M_{22}$ is not power-chordal.}
\end{theorem}
\begin{proof}
	As per GAP and Sage \cite{ GAP, Sage}, we observe that $M_{22}$ contains $1155$ elements of order $2$, $12320$ elements of order $3$ and $36960$ elements of order $6$. So a subgroup of order $2$ is contained in $16$ distinct cyclic subgroup of order $6$, and a subgroup of order $3$ is contained in $3$ distinct cyclic subgroup of order $6$. Thus, we can choose elements $a$, $b$, $c$, and $d$, each of order $6$, with $a^{3}=b^{3}$, $b^{2}=c^{2}$, $c^{3}=d^{3}$ and $d^{2}=a^{2}$. Then $\mathcal{P}(M_{22})$ contains an $8$-cycle $(a, a^{3}, b, b^{2}, c, c^{3}, d, d^{2}, a)$, which ensures the non-chordality of $\mathcal{P}(M_{22})$.
\end{proof}

\begin{theorem}
\label{chordal_sps_th4}
{$J_{1}$ is not power-chordal.}
\end{theorem}
\begin{proof}
By the information at \cite{Conway}, we observe that $J_{1}$ contains $3$ elements say $x, y, z$ of orders $6, 10$ and $15$ respectively such that $x^{2}=z^{5}, x^{3}=y^{5}$ and $y^{2}=z^{3}$. Then $\mathcal{P}(J_{1})$ contains a cycle $(x^{2}, x, x^{3}, y, y^{2}, z, x^{2})$, whose length is $6$. Thus, the Janko group $J_{1}$ is not a power-chordal group.

\end{proof}

\begin{theorem}
\label{chordal_sps_th3}
{$J_3$ is not power-chordal.}
\end{theorem}
\begin{proof}
We observe that the Janko group $J_3$ contains elements $a$, $b$, and $c$ such that $o(a)=6$, $o(b)=10$, $o(c) = 15$, and $a^3=b^5$, $a^2=c^5$, $b^2=c^3$. Therefore, $(a^2,a,a^3,b,b^2,c,c^5)$ is a $6$-cycle in the power graph of $J_3$. Therefore, $\mathcal{P}(J_3)$ is not a chordal graph.
\end{proof}
\begin{theorem}
{There are not power-chordal sporadic groups.}
\end{theorem}
\begin{proof}
We have already proved that $\mathcal{P}(M_{11})$, $\mathcal{P}(M_{22})$, $\mathcal{P}(J_{1})$ and $\mathcal{P}(J_3)$ are not chordal. The rest of the sporadic groups always contain a non-power-chordal supgroup (see Table \ref{Table_1}). This completes the proof of the theorem.
\begin{table}[h]
\centering
\begin{tabular}{|c|c|c|}
\hline
\textbf{sporadic simple Groups}&  \textbf{non-power-chordal subgroup(s)} &\textbf{remarks}\\ 
\hline
$M_{12}$ & $M_{11}$ & Theorem \ref{chordal_sps_th1} \\
\hline
$M_{23}, M_{24}$ & $A_{8}$& Theorem \ref{chordal_alternating} \\
\hline
$Fi_{22}$ & $S_{10}$& Theorem \ref{Chordal_symm_th1} \\
\hline
$Fi_{23}$ & $S_{12}$& Theorem \ref{Chordal_symm_th1} \\
\hline
$Fi_{24}$ & $Fi_{23}$, $S_{12}$& Theorem \ref{Chordal_symm_th1} \\
\hline
$J_{2}$ & $\PSU(3, 3)$& Theorem \ref{PSU_odd}\\
\hline
$J_{4}$ & $C_{2^{11}}\times M_{24}$, $A_8$& Theorem \ref{chordal_alternating}\\
\hline
$Co_{2}, Co_{3}$ &  $M_{23}$, $A_8$ & Theorem \ref{chordal_alternating}\\
\hline
$Co_{1}$ & $Co_{2}, Co_{3}$, $A_8$& Theorem \ref{chordal_alternating}\\
\hline
$He$ & $S_{7}$& Theorem \ref{Chordal_symm_th1}\\
\hline
$HS$ & $S_{8}$& Theorem \ref{Chordal_symm_th1}\\
\hline
$O'N$ & $J_{1}$& Theorem \ref{chordal_sps_th3}\\
\hline
$HN$ & $A_{12}$& Theorem \ref{chordal_alternating}\\
\hline
$Ly$ & $A_{11}$& Theorem \ref{chordal_alternating}\\
\hline
$Th$ & $C_{3}\times G_{2}(3)$& Theorem \ref{G_{2}(q)}\\
\hline
$B$ & $(S_{3}\times Fi_{22})\rtimes C_{2}$, $S_{10}$ & Theorem \ref{Chordal_symm_th1}\\
\hline
$M$ & $A_{5}\times A_{12}$& Theorem \ref{chordal_alternating}\\
\hline
$M^{c}L$ & $M_{22}$,& Theorem \ref{chordal_sps_th2}\\
\hline
\end{tabular}
\label{Table_1}
\caption{Non-power-chordal subgroups of sporadic simple groups}
\end{table}
\end{proof}

\section{ Groups of small order}
In this section, we discuss certain finite groups whose power graph belongs to the class of chordal graphs. We already know that, for any prime $p$, the power graph of a group of order  $p^\alpha$ or $2p$ is chordal. In this section, we prove that groups of order $pq$ and $p^{2}q$ are power-chordal. Further, we compute the values of $n$ for which the dicyclic group $\mathrm{Dic}_{n}=\langle a, x| a^{2n}=1, x^{2}=a^{n}, x^{-1}ax=a^{-1} \rangle$ is power-chordal. In addition, we provide a table that describes all power-chordal groups of order up to 47.

\begin{theorem}
\label{Chordal_sg_th1}
Let $G$ be a finite group of order $pq$, where $p,q$ are distinct primes. Then $\mathcal{P}(G)$ is a chordal graph.
\end{theorem}

\begin{proof}
{Assume that} $p>q$. Then there are at most two groups of order $pq$; one is the cyclic group $C_{pq}$ and the other is $C_{p} \rtimes C_{q}$ if $q|(p-1)$. Now, $C_{pq}$  has  a chordal power graph by Theorem \ref{Chordal_pre_th5}. {Note that} $C_{p} \rtimes C_{q}$ is an EPPO group, so it is power-chordal.
\end{proof}

\begin{theorem}
\label{Chordal_dic_th1}
The dicyclic group $\mathrm{Dic}_{n}$ is a power-chordal group if and only if $n$ is either a power of a prime or a product of an odd prime and a power of $2$.
\end{theorem}

\begin{proof}
First, suppose that $\mathcal{P}(\mathrm{Dic}_{n})$ is chordal. Note that $\mathrm{Dic}_{n}$ contains a unique maximal cyclic subgroup, say $H$, of order $2n$ and two other maximal subgroups, say $L$ and $M$, of order $2n$, and the remaining subgroups are of order $4$. Let $a$ and $x$ be two generators of $Dic_{n}$, with the order of $x$ being $4$. Suppose $H=\langle a\rangle$, $L=\langle a^{2}, x \rangle$, and $M=\langle a^{2}, ax\rangle$. Clearly, $a^n$ is the unique involution in $G$, and so $a^n$ is contained in every subgroup of $\mathrm{Dic}_n$. {So the identity and the involution cannot occur in any induced chordless cycle of $\mathcal{P}(\mathrm{Dic}_{n})$.} Let $\mathcal{P}^{**}(\mathrm{Dic}_{n})$ denote the graph obtained by removing the identity and the unique involution from $\mathcal{P}(\mathrm{Dic}_{n})$. Now we show that if $\mathcal{P}^{**}(\mathrm{Dic}_{n})$ contains any induced cycle, it must be contained in $\mathcal{P}^{**}(C_{2n})$. 

First, we look at the maximal subgroups $L$ and $M$. We observe that $\mathcal{P}^{**}(L)$ is a disjoint union of $\mathcal{P}^{**}(C_{n})$ along with some copies of $K_{2}$. Also, $\mathcal{P}^{**}(M)$ has a similar structure. Moreover, $L, M, H$ have a common intersection, which is the subgroup generated by $\langle a^{2}\rangle$. All the other subgroups (except $\langle x \rangle$ and $\langle ax \rangle$) of $\mathrm{Dic}_{n}$, whose orders are $4$, contribute disjoint unions of $K_{2}$ in $P^{**}(\mathrm{Dic}_{n})$; for the exceptional cases, these two subgroups form $K_{2}$ in $L, M$ resp. Note that if there is any vertex from $L$ or $M$ in an induced cycle, they must belong to the cyclic subgroup generated by $\langle a^{2}\rangle$ (which is isomorphic to $C_{n}$). Since $\langle a^{2}\rangle$ is also contained in $H$ this implies that the whole cycle must be contained in $\mathcal{P}^{**}(C_{2n})$. 

 Thus, we conclude that $\mathcal{P}(\mathrm{Dic}_{n})$ is a chordal graph if and only if the power graph of $H$ is chordal. Hence the proof of the theorem follows by Corollary \ref{Chordal_pre_cor1}.
\end{proof}

\begin{theorem}
\label{Chordal_sg_th2}
Any group of order $p^{2}q$, where $p$ and $q$ are distinct primes, is power-chordal.
\end{theorem}

\begin{proof}
Let $G$ be a group of order $p^2q.$ If $G$ is abelian, then either $G\cong C_{p^{2}q}$ or $G\cong C_{p} \times C_{pq}$. So, by Theorem \ref{Chordal_pre_th5} and Corollary \ref{Chordal_pre_cor1}, $\mathcal{P}(G)$ is a chordal graph.\medskip

Next, suppose that $G$ is non-abelian. First, we assume that $p>q$. Then the Sylow $p$-subgroup, say $H$, is normal in $G$. Then either $G\cong C_{p} \times (C_{p} \rtimes C_{q})$ or $G\cong C_{p^2} \rtimes C_{q}$. Now for these groups, we have the following observations 

If $G$ does not contain any element of order $pq$, then its prime graph is a null graph, and hence $G$ is a power-chordal group. Otherwise, $G$ has an element of order $pq$. Now if $G\cong C_{p}\times (C_{p}\rtimes C_q)$, then the chordality of $\mathcal{P}(G)$ follows from Lemma \ref{Chordal_product_th1}. Next, let $G\cong C_{p^2}\rtimes C_{q}$. If possible, let $\cdots x_{1}\rightarrow x_{2}\leftarrow x_{3}\rightarrow x_{4}\leftarrow\cdots$ be a cycle in $\overrightarrow{\mathcal{P}}(G)$. Now every element of $G$ is of order either $q$ or a power of $p$ or $pq$. Consider the element $x_3$. If $o(x_3)$ is a power of $p$ or $q$, then $x_2\sim x_4$.

Next, let $o(x_3)=pq$. Then either $o(x_{2})=q$, $ o(x_{4})=p$ or $o(x_{2})=p$, $o(x_{4})=q$. Without loss of generality, we assume that $o(x_{2})=q$ and $ o(x_{4})=p$. Now $o(x_{1})$ must be either $q$ or $pq$. If $o(x_1)=q$, then $x_{3}\sim x_{1}$. Again, if $o(x_5)=pq$, then the Sylow $p$-subgroup of $G$ is unique and cyclic, and $x_{2}\leftarrow x_{1}$, and so $x_{1}\in \langle c\rangle$. Hence, we obtain the chord $x_{3}\sim x_{1}$. Thus, a cycle of length $\geq4$ is not possible in $\mathcal{P}(G)$. Therefore, $\mathcal{P}(G)$ is chordal. \medskip

Next, {suppose} $p<q$. Clearly, by Burnside's theorem, $G$ is not simple. So it must contain a non-trivial normal subgroup. Here arise two cases.\medskip\\
\textbf{Case I:} Let $G$ contain a normal subgroup of order $pq$, say $K$; if $K$ is cyclic, then $G\cong C_{p} \times (C_{q} \rtimes C_{p})$. Then $\mathcal{P}(G)$  is chordal by Lemma \ref{Chordal_product_th1}. Otherwise, $K$ must be the non-abelian group $C_{q}\rtimes C_{p}$ of order $pq$, and hence $G\cong (C_{q}\rtimes C_{p}) \rtimes C_{p}$. Then, by a similar approach as above, we can show that an induced cycle of length at least $4$ is not possible in $\mathcal{P}(G)$. Hence, $\mathcal{P}(G)$ is chordal.\medskip\\
\textbf{Case II:} If $G$ does not contain any subgroup of order $pq$.
Then the GK-graph of $G$ is edgeless, and so its power graph is  a chordal graph by Corollary \ref{Chordal_pre_cor4}.

Thus, if $G$ is a group of order $p^2q$, then $\mathcal{P}(G)$ is a chordal graph.\end{proof}

Next, we create a table that lists the power-chordal groups of order up to $47$ by using  Theorem \ref{Chordal_sg_th1}, \ref{Chordal_dic_th1}, \ref{Chordal_sg_th2} and the results obtained in Section \ref{direct products}.
\begin{table}[h]
\centering
\begin{tabular}{|c|c|}
\hline
{Order of $G$}&  {Groups} \\ \hline
$24$ & all groups except $C_{2}\times C_{12}$, $C_{3}\times Q_{8}$ and $C_{3}\times D_{4}$\\
\hline
$30$ & all non-cyclic groups \\
\hline
$36$ & all groups except $ C_{36}, C_{6} \times C_{6}$\\ &and  $C_{2}\times (C_{3}\rtimes S_{3})$\\
\hline
$40$ & all groups except $C_{2} \times C_{20}$, $C_{5}\times D_{4}$ \\ & and $C_{5}\times Q_{8}$\\
\hline
$42$ & all non-cyclic groups \\
\hline
$|G|\leq47$ and $|G|\notin \{24, 30, 36, 40, 42\}$ & all groups \\
\hline
\end{tabular}
\caption{Power-chordal groups of order {up to} $47$}
\end{table}

\begin{remark}
For the small groups of order at most $47$, we obtain that almost all the groups are power-chordal except a few groups with the orders $24, 30, 36$, $40$, and $42$.
\end{remark}

\subsection*{Acknowledgements}
The author Pallabi Manna is grateful to CSIR(Grant No- 09/983(0037)/2019-EMR-I) for financial aid. Ranjit Mehatari thanks the SERB, India, for financial support (File Number: CRG/2020/000447) through the Core Research Grant.

\end{document}